\documentclass[a4paper,10pt]{article}
\usepackage[utf8]{inputenc}
\usepackage{amsmath}
\usepackage{amsfonts}
\usepackage{amssymb}
\usepackage{cite}
\usepackage{enumerate}
\usepackage{fancybox}
\usepackage{fullpage}

\RequirePackage{hyperref}
\PassOptionsToPackage{pdftex,bookmarksopen,bookmarksnumbered}{hyperref}
\usepackage{hyperref}
\usepackage{fullpage}

\usepackage{amsthm} 
\newtheorem{thm}{Theorem}[section]

\newtheorem{lemma}[thm]{Lemma}
\newtheorem{proposition}[thm]{Proposition}
\newtheorem{cor}[thm]{Corollary}
\newtheorem{corollary}[thm]{Corollary}
\newtheorem{propn}[thm]{Proposition}
\newtheorem{defn}[thm]{Definition}
\newtheorem{definition}[thm]{Definition}
\newtheorem{eg}[thm]{Example}

\newtheorem{remark}[thm]{Remark}

\renewcommand{\Bbb}{\mathbb{B}}

\newcommand{\Ebb}{\mathbb{E}}

\newcommand{\Nbb}{\mathbb{N}}

\newcommand{\Rbb}{\mathbb{R}}

\newcommand{\Xbb}{\mathbb{X}}
\newcommand{\Ybb}{\mathbb{Y}}

\newcommand{\xhat}{{\widehat{x}}}

\newcommand{\ubar}{{\overline{u}}}
\newcommand{\vbar}{{\overline{v}}}
\newcommand{\wbar}{{\overline{w}}}
\newcommand{\xbar}{{\overline{x}}}
\newcommand{\ybar}{{\overline{y}}}

\newcommand{\xtilde}{{\widetilde{x}}}

\newcommand{\varepsilontilde}{{\widetilde{\varepsilon}}}

\newcommand{\Euclid}{\mathbb{E}}

\newcommand{\paren}[1]{\left(#1\right)}

\renewcommand{\equiv}{:=}

\newcommand{\Ball}{{\Bbb}}

\newcommand{\Rtw}{{\Rbb^2}}

\newcommand{\set}[2]{\left\{#1\,\left|\,#2\right.\right\}}
\newcommand{\map}[3]{#1:\,#2\rightarrow #3\,}
\newcommand{\mmap}[3]{#1:\,#2\rightrightarrows #3\,}

\newcommand{\ip}[2]{\left\langle #1,~ #2\right\rangle}

\newcommand{\norm}[1]{\left\|#1\right\|}

\newcommand{\und}{\quad\mbox{and}\quad}

\DeclareMathOperator{\Id}{Id}

\DeclareMathOperator{\argmin}{argmin\,}

\DeclareMathOperator{\dist}{dist}

\DeclareMathOperator{\cone}{{cone}}

\DeclareMathOperator{\gph}{gph}

\DeclareMathOperator{\Fix}{\mathsf{Fix}\,}
\DeclareMathOperator{\epi}{{\rm epi}}

\newcommand{\ncone}[1]{{N}_{#1}}

\newcommand{\fncone}[1]{{\widehat{N}}_{#1}}
\newcommand{\pncone}[1]{N^{\text{\rm prox}}_{#1}} 

\usepackage[shortlabels]{enumitem}
\newcommand {\R} {\mathbb R}

\newcommand {\B} {\mathbb B}

\newcommand{\bx}{\overline x}

\newcommand {\bd} {{\rm bd}\,}

\newcommand{\sr}{{\rm sr}[A,B](\bx)}

\newcommand{\srr}{{\rm sr}'[A,B](\bx)}
\newcommand{\srrt}{{\rm sr}'[A,B](\xtilde)}

\newcommand{\TO}[1]{\stackrel{#1}{\to}}
\newcommand {\Limsup} {\mathop{{\rm Lim\,sup}\,}}

\title{Necessary conditions for linear convergence of iterated expansive, set-valued mappings with application to alternating projections}
\author{
D. Russell Luke\thanks{Institut f\"ur Numerische und Angewandte Mathematik,
Universit\"at G\"ottingen,
37083 G\"ottingen, Germany. DRL was supported in part by German Israeli Foundation Grant G-1253-304.6 and
Deutsche Forschungsgemeinschaft Research Training Grant 2088 TP-B5.
E-mail:  \texttt{r.luke@math.uni-goettingen.de}}
,~
Marc Teboulle\thanks{School of Mathematical Sciences, Tel Aviv University,
Tel Aviv 69978, Israel. MT was supported by German Israeli Foundation Grant G-1253-304.6 and Israel Science Foundation ISF Grant 1844-16.
E-mail:  \texttt{teboulle@post.tau.ac.il.}}
~and 
Nguyen H. Thao\thanks{Delft Center for Systems and Control, Delft University of Technology, 2628CD Delft, The Netherlands.
Department of Mathematics, Teacher College, Cantho University,
Cantho, Vietnam. NHT was supported by German Israeli Foundation Grant G-1253-304.6.
E-mail:  \texttt{h.t.nguyen-3@tudelft.nl, h.nguyen@math.uni-goettingen.de, nhthao@ctu.edu.vn}}
}
\date{April 28, 2017}
\begin{document}

\maketitle

\begin{abstract}
We present necessary conditions for monotonicity, in
one form or another, of  fixed point iterations  of mappings that violate the usual nonexpansive property.
We show that most reasonable notions of linear-type
monotonicity of fixed point sequences imply {\em metric subregularity}.
This is specialized to the alternating projections iteration where
the metric subregularity property takes on a distinct geometric characterization of sets at
points of intersection called {\em subtransversality}.  Our more general results for fixed point iterations are specialized to establish the necessity
of subtransversality for consistent feasibility with a number of reasonable types of sequential monotonicity, under
varying degrees of assumptions on the regularity of the sets.
\end{abstract}

{\small \noindent {\bfseries 2010 Mathematics Subject
Classification:} {Primary 49J53, 65K10
Secondary 49K40, 49M05, 49M27, 65K05, 90C26.\\
}}


\section{Introduction}
In recent years there has been a lot of progress in determining ever weaker conditions to guarantee
local linear convergence of elementary fixed point algorithms, with particular attention given
to the method of alternating projections and the Douglas-Rachford iteration
\cite{LewisMalick08, LewLukMal09, DruIofLew15, HesLuk13, BauLukePhanWang13b, NolRon16, Phan16,
LukNguTam17, Giselsson15}.
These works beg the question:  what are necessary conditions for linear convergence?  We
shed some light on this question for iterations generated by not necessarily nonexpansive fixed point mappings and show how our theory
specializes for the alternating projections iteration in nonconvex and convex settings.

Our work builds upon the terminology and theoretical framework established in \cite{LukNguTam17}.
As much as possible, we have tried to make the present analysis self-contained, but it is not
possible to reproduce all the results taken from \cite{LukNguTam17}.
After introducing basic notation and definitions in Section \ref{s:notation},  we clarify first what
we mean by linear convergence, since there are many ways in which
a sequence can behave linearly with respect to the set of fixed points.  We introduce a generalization of
Fej\'er monotonicity, namely {\em linear monotonicity} (Definition \ref{d:mu_Mon}) which is central
to our development.   We also introduce
another generalization, {\em linearly extendible sequences} in
Definition \ref{Uni_Lin_Con_to_point}, that concerns sequences which can be viewed as the subsequence
of some monotone sequence.  This is key to the application to alternating projections studied in Sections
\ref{s:AP_for_NonConvex} and \ref{s:AP_with_Convex}.    In
Section \ref{s:lmfp} we lay the groundwork for the first main result on necessary
conditions for linearly monotone fixed point iterations with respect
to $\Fix T$ (Theorem \ref{t:Nec_for_msr}).  The result states that {\em metric subregularity}
(Definition \ref{d:(str)metric (sub)reg}) is necessary for linearly monotone fixed point iterations.
If in addition the fixed point operator $T$ is {\em almost averaged} at points in $\Fix T$ (Definition \ref{d:pane/paa}), then
metric subregularity is necessary for {\em linear convergence} of the iterates to
a point in $\Fix T$ (Corollary \ref{t:nec_suff}).
Sections \ref{s:AP_for_NonConvex} and \ref{s:AP_with_Convex} are specializations to the case of alternating
projections for {\em consistent feasibility}.  In this setting metric subregularity takes on the more directly geometric interpretation as
{\em subtransversality} of the sets at common points (Definition \ref{d:MAINi}).  Theorem \ref{NonCon_Nec1+} establishes the
necessity of subtransversality for alternating projections iterations to be linearly monotone with respect to common points.
Corollary \ref{Nec_Suf_Sub} then shows that for sets with a certain {\em elemental subregularity} (Definition \ref{d:set regularity})
subtransversality is necessary and sufficient for linear monotonicity of the sequence.  For sequences that are R-linearly
convergent to a fixed point and satisfy a subsequential linear monotonicity property (condition \eqref{k_1}),
Theorem \ref{NonCon_Nec2} shows that subtransversality is also necessary.
Subtransversality is also shown to be necessary
for sequences to have linearly extendible subsequences (Theorem \ref{NonCon_Nec1+_part2}).  These results
correspond to our observation in Proposition \ref{Sub_appear} that
subtransversality has appeared in one form or another in all sufficient conditions for linear monotonicity or convergence of
alternating projections for consistent feasibility that have appeared previously in the literature.
In Section \ref{s:AP_with_Convex}
these results are further specialized to the case of convex feasibility.  We show in Theorems \ref{BasThe_Loc} and
\ref{BasThe_Glo} that metric subregularity of some form is necessary and sufficient for local and global linear convergence
of alternating projections.  Moreover, we show in Proposition \ref{u_R} that R-linear convergence of the sequences in this setting is equivalent to
linear monotonicity of the sequence with respect to points of intersection. For
 Q-linear convergence, we show that
linear extendability is necessary (Proposition \ref{Q_R}).

Based on the results obtained here we conjecture that, for alternating projections applied to {\em inconsistent} feasibility,
subtransversality as extended in \cite[Definition 3.2]{KruLukNgu17} is also necessary for R-linear convergence of the 
iterates to fixed points.

\section{Notation and basic definitions}\label{s:notation}
Throughout our discussion $\Euclid$ is a Euclidean space.
Given a subset $A\subset\Euclid$, $\dist(x,A)$ stands for the distance from a point $x\in \Euclid$ to $A$:
$\dist(x,A) := \inf_{a\in A}\|x-a\|$.
The {\em projector} onto the set $A$, $\mmap{P_{A}}{\Euclid}{A}$, is central to algorithms for feasibility and is defined by
\[
P_{A}x\equiv \underset{a\in A}{\argmin}\norm{a-x}.
\]
A {\em projection} is a selection from the projector.
This exists for any closed set $ A\subset\Euclid$, as can be deduced by the continuity and coercivity of the norm.
Note that the projector is not, in general, single-valued, and indeed uniqueness of the projector defines a type of 
regularity of the set $A$:
local uniqueness characterizes {\em prox-regularity} \cite{PolRocThi00}  while in finite dimensional settings 
global uniqueness characterizes
convexity \cite{Bunt}.

Given a subset $A\subset\mathbb{E}$ and a point $\bx\in  A$, the
\emph{Fr\'echet, proximal and limiting normal cones} to $ A$ at $\bx$ are defined, respectively, as follows:
\begin{gather*}
\fncone{A}(\xbar):= \left\{v\in\Euclid\mid \limsup_{x\TO{A}\xbar,\,x\neq \xbar} \frac{\langle v,x-\xbar \rangle}{\|x-\xbar\|} \leq 0 \right\},
\\
\pncone{A}(\bx):=\cone\left(P_{A}^{-1}(\bx)-\bx\right),
\\
\ncone{A}(\xbar):= \Limsup_{x\TO{A}\xbar} \pncone{A}(x):
=\left\{v=\lim_{k\to\infty}v_k\mid v_k\in \pncone{A}(x_k),\; x_k\TO{A} \xbar \right\}.
\end{gather*}
In the above, $x\TO{ A} \bx$ means that $x\to \bx$ with $x\in  A$.

Our other basic notation is standard; cf. \cite{Mord06,VA,DontRock09}.
The open unit ball
in a Euclidean space
is denoted $\B$.
$\B_\delta(x)$ stands for the open ball with radius $\delta>0$ and center $x$;  $\B_\delta$
  is the open ball of radius $\delta$ centered at the origin.

To quantify convergence of sequences and fixed point iterations, we focus primarily on 
linear convergence, though sublinear convergence
can also be handled in this framework.
Linear convergence, however, can come in many forms.  We list the more common notions next.

\begin{defn}[R- and Q-linear convergence to points, Chapter 9 of \cite{OrtegaRheinboldt70}]\label{d:q-r-lc}
Let $(x_k)_{k\in\Nbb}$ be a sequence in $\Ebb$.
\begin{enumerate}
\item $(x_k)_{k\in\Nbb}$ is said to \emph{converge R-linearly} to $\xtilde$ with rate $c\in [0,1)$ 
if there is a constant $\gamma>0$ such that
\begin{align}\label{LinCon_Seq}
\norm{x_k-\xtilde} \le \gamma c^k\quad \forall k\in \Nbb.
\end{align}
\item $(x_k)_{k\in\Nbb}$ is said to \emph{converge Q-linearly} to $\xtilde$ with rate $c\in [0,1)$ if
\[
\norm{x_{k+1}-\xtilde} \le c\norm{x_{k}-\xtilde}\quad \forall k\in \Nbb.
\]
\end{enumerate}
\end{defn}

By definition, Q-linear convergence implies R-linear convergence with the same rate. Elementary examples show
that the inverse implication does not hold in general.

One of the central concepts in the convergence of sequences is {\em Fej\'er monotonicity}:
a sequence $(x_k)_{k\in\Nbb}$ is \emph{Fej\'er monotone} with respect to a nonempty convex set $\Omega$ if
\[
\|x_{k+1}-x\| \le \|x_k-x\|\quad \forall x\in \Omega,\, \forall k\in \mathbb{N}.
\]
\noindent In the context of convergence analysis of fixed point iterations,
  the following generalization of Fej\'er monotonicity of sequences is central.

\begin{defn}[$\mu$-monotonicity]\label{d:mu_Mon}
Let $(x_k)_{k\in\Nbb}$ be a sequence on $\Ebb$, $\Omega\subset\Ebb$ be nonempty and
$\map{\mu}{\Rbb_+}{\Rbb_+}$ satisfy $\mu(0)=0$ and
\begin{align*}
&\mu(t_1)<\mu(t_2)\leq t_2\; \mbox{ whenever }\; 0\le t_1<t_2.
\end{align*}
\begin{enumerate}[(i)]
 \item $(x_k)_{k\in \Nbb}$ is said to be
 \emph{$\mu$-monotone with respect to $\Omega$}  if
 \begin{equation}\label{e:mu-uniform mon}
 \dist(x_{k+1}, \Omega)\leq \mu\paren{\dist(x_k, \Omega)}\quad  \forall k\in \Nbb .
 \end{equation}
 \item $(x_k)_{k\in \Nbb}$ is said to be
 \emph{linearly monotone with respect to $\Omega$}  if \eqref{e:mu-uniform mon} is
 satisfied for $\mu(t)=c\cdot t$ for all $t\in \Rbb_+$ and some constant $c\in [0,1]$.
\end{enumerate}
\end{defn}
The focus of our study is linear convergence, so only linear monotonicity will come into play
in what follows.  A study of other kinds of convergence, particularly {\em sublinear}, would
employ the full generality of $\mu$-monotonicity.

The next result is clear.

\begin{propn}[Fej\'er monotonicity implies $\mu$-monotonicity]\label{t:Fejer-uniform}
 If the sequence $(x_k)_{k\in \Nbb}$ is Fej\'er monotone with respect to $\Omega\subset\Xbb$
 then it is $\mu$-monotone with respect to $\Omega$ with $\mu=\Id$.
\end{propn}
The converse is not true, as the next example shows.

\begin{eg}[$\mu$-monotonicity is not Fej\'er monotonicity]
 Let $\Omega\equiv \set{(x,y)\in\Rtw}{y\leq 0}$ and consider the sequence
 $x_k:=\paren{1/2^{k},1/2^{k}}$ for all $k\in \Nbb$.
This sequence is linearly monotone with respect to $\Omega$ with constant $c=1/2$,
but not Fej\'er monotone since $\|x_{k+1}-(2,0)\|>\|x_{k}-(2,0)\|$ for all $k$.
\end{eg}	

The next definition will come into play in Sections \ref{s:AP_for_NonConvex} and
\ref{s:AP_with_Convex}.  It provides a way to
analyze fixed point iterations  which, like our main example of alternating projections,
are compositions of mappings.

The subset $\Lambda\subset \Euclid$ appearing in Definition \ref{Uni_Lin_Con_to_point} and throughout
this work is always assumed to be closed and nonempty.
We use this set to isolate specific elements of the fixed point set
(most often restricted to affine subspaces).  This is more
than just a formal generalization since in some concrete situations
the required assumptions do not hold on $\Ebb$ but they do hold on
relevant subsets.

\begin{definition}[linearly extendible sequences]\label{Uni_Lin_Con_to_point}
A sequence $(x_k)_{k\in\Nbb}$ on $\Lambda\subset\Ebb$ is said to be
\emph{linearly extendible} on $\Lambda$ with frequency $m\ge 1$ ($m\in \Nbb$ is fixed) and rate $c\in [0,1)$ if
there is a sequence $(z_k)_{k\in\Nbb}$ on $\Lambda$
such that $x_k=z_{mk}$ for all $k\in \mathbb{N}$ and the following conditions are satisfied for all $k\in \Nbb$:
\begin{align}\label{LinCon_Seq+'_m}
\|z_{k+2}-z_{k+1}\|\; &\le\; \|z_{k+1}-z_{k}\|,
\\ \label{LinCon_Seq+'_m2}
\|z_{m(k+1)+1}-z_{m(k+1)}\|\; &\le\; c\|z_{mk+1}-z_{mk}\|.
\end{align}

When $\Lambda=\Ebb$, the quantifier ``on $\Lambda$'' is dropped.
\end{definition}
The requirement on the linear extension sequence $(z_k)_{k\in\Nbb}$ means
that the sequence of the distances between its two consecutive iterates
 is uniformly non-increasing and possesses a subsequence of type
$(\|z_{mk+1}-z_{mk}\|)_{k\in \Nbb}$ that converges Q-linearly with a global rate
to zero.

The extension of sequences of fixed point iterations
  $(x_k)_{k\in\Nbb}$ will most often be to the intermediate
points generated by the composite mappings.  In the case of alternating projections this is
$z_{2k}\equiv x_k\in P_AP_Bx_{k-1}$, and $z_{2k+1}\in P_B z_{2k}$. This strategy of analyzing alternating
projections by keeping track of the intermediate projections has been exploited to great effect in
\cite{LewisMalick08,LewLukMal09,BauLukePhanWang13b,
NolRon16,DruIofLew15, LukNguTam17}.
From the Cauchy
property of $(z_k)_{k\in\Nbb}$, one can deduce
R-linear convergence from linear extendability.

\begin{propn}[linear extendability implies R-linear convergence]\label{t:Com_LinCon_point}
If the sequence $(x_k)_{k\in\Nbb}$ on $\Lambda\subset \Ebb$ is linearly extendible on $\Lambda$ with
some frequency $m\ge 1$ and rate $c\in [0,1)$, then $(x_k)_{k\in\Nbb}$ converges R-linearly
to a point $\xtilde \in \Lambda$ with rate $c$.
\end{propn}

\begin{proof}
Let $(z_k)_{k\in\Nbb}$ be a linear extension of $(x_k)_{k\in\Nbb}$ on $\Lambda$ with frequency $m$ and rate $c$.
Conditions \eqref{LinCon_Seq+'_m} and \eqref{LinCon_Seq+'_m2} then imply
by induction that
\begin{equation*}
\|z_{k+1}-z_k\|\le \frac{d_0}{c} \sqrt[m]{c}^k\quad \forall k\in\Nbb,
\end{equation*}
where $d_0:=\|z_1-z_0\|$.
This means that $(z_k)_{k\in\Nbb}$ is a Cauchy sequence and hence converges to a limit
$\xtilde$, which is in $\Lambda$ by the closedness of this set.
Conditions \eqref{LinCon_Seq+'_m} and \eqref{LinCon_Seq+'_m2} also yield that
for every $k\in \Nbb$,
\begin{align*}
&\|x_k-\xtilde\| = \|z_{mk}-\xtilde\|
\le \sum_{i=mk}^{\infty}\|z_{i}-z_{i+1}\|
\\
\le\; & m\sum_{i=k}^{\infty}\|z_{mi}-z_{mi+1}\|
\le m\|z_{0}-z_{1}\|\sum_{i=k}^{\infty}c^i
\le \gamma c^k,
\end{align*}
where
\begin{equation}\label{gamma}
  \gamma:=\frac{md_0}{1-c}.
\end{equation}
This means that $(x_k)_{k\in\Nbb}$ converges R-linearly to $\xtilde$ with rate $c$.
\end{proof}

\section{Linearly monotone fixed point iterations}\label{s:lmfp}

Quantifying the convergence of
fixed point iterations is key to providing error bounds for algorithms.
For a multi-valued self-mapping $\mmap{T}{\Ebb}{\Ebb}$, the operative
inequality leading to linear convergence of the
fixed point iteration
$x_{k+1}\in Tx_k$ is
\begin{align}\label{LinCon_Seq+}
\dist(x_{k+1},S)\; \le\; c\dist(x_{k},S) \quad \forall k\in \Nbb
\end{align}
for $S\subset \Fix T$ and $c\in [0,1)$.
When this holds, the sequence $\paren{x_k}_{k\in\Nbb}$ is linearly monotone with
respect to $S$ and
constant $c$.

For multi-valued mappings, however, we need to clarify what is meant in the first place by the fixed point set.
We take the least restrictive definition as any point contained in its image via the mapping.
\begin{defn}[fixed points of set-valued mappings]\label{d:fixed points}
The set of fixed points of a possibly set-valued mapping $\mmap{T}{\Ebb}{\Ebb}$ is defined by
\[
\Fix T\equiv \set{x\in \Ebb}{x\in Tx}.
\]
\end{defn}
As noted in \cite[Example 2.1]{LukNguTam17}, for $x\in \Fix T$, it is not the case that $Tx\subset \Fix T$.
This can happen, in particular, when the mapping $T$ is multi-valued on its set of fixed points.
Almost averaged mappings detailed next are a generalization of averaged mappings and rule out so-called inhomogeneous
fixed point sets.

\subsection{Almost averaged mappings}
\begin{defn}[almost nonexpansive/averaged mappings, Definition 2.2 of \cite{LukNguTam17}]\label{d:pane/paa}
Let $\Omega$ be a nonempty subset of $\Ebb$ and let $\mmap{T}{\Omega}{\Ebb}$.
\begin{enumerate}[(i)]
   \item\label{d:pane}  $T$ is said to be  {\em pointwise almost
nonexpansive on $\Omega$ at $ y \in \Omega$}
if there exists a $\varepsilon\ge 0$ (called the {\em violation}) such that
\begin{equation}\label{e:epsqnonexp}
	\begin{aligned}
&\norm{x^+ - y^+}\leq\sqrt{1+\varepsilon}\norm{x- y}\\
&\forall x\in \Omega \quad \forall x^+\in Tx\quad \forall y^+\in Ty.
   \end{aligned}
\end{equation}
If \eqref{e:epsqnonexp} holds with $\varepsilon=0$ then $T$ is called
\emph{pointwise nonexpansive} at $y$ on $\Omega$.

If $T$ is pointwise (almost) nonexpansive  on $\Omega$ at every point on a
neighborhood
of $y$ in $\Omega$ (with the same violation
$\varepsilon$), then $T$ is said to be
\emph{(almost) nonexpansive on $\Omega$ at $y$ (with violation
$\varepsilon$)}.

If $T$ is pointwise (almost) nonexpansive at every point $y\in \Omega$
(with the same violation
$\varepsilon$) on $\Omega$, then $T$ is said to be
\emph{(almost) nonexpansive on $\Omega$ (with violation
$\varepsilon$)}.

\item\label{d:paa} $T$ is called \emph{pointwise almost averaged on $\Omega$ at $y\in \Omega$} with violation
$\varepsilon\ge 0$ if
there is an averaging constant $\alpha\in (0,1)$
such that the mapping
$\widetilde{T}$ defined by
\[
\widetilde{T}\equiv  \tfrac{1}{\alpha}T - \tfrac{\paren{1-\alpha}}{\alpha}\Id
\]
is pointwise almost nonexpansive  on $\Omega$ at $y$ with violation $\varepsilon/\alpha$.

Likewise $T$ is said to be
\emph{(pointwise) (almost) averaged  on $\Omega$ (at $y$) (with violation $\alpha\varepsilon$)}
if $\widetilde{T}$ is (pointwise) (almost) nonexpansive  on $\Omega$ (at $y$) (with violation
$\varepsilon$).
\end{enumerate}
\end{defn}
\begin{remark}
The following remarks help to place this property in context.
\begin{enumerate}[(a)]
   \item A mapping $T$ is averaged with violation $\varepsilon=0$ and averaging constant $\alpha=1/2$
at all points on $\Omega$ if and only if it is firmly nonexpansive on $\Omega$.
\item As noted in \cite{LukNguTam17}, pointwise almost nonexpansiveness of $T$ at $\xbar$ with violation $\varepsilon$
is related to, but stronger than {\em calmness} \cite[Chapter 8.F]{VA} with constant 
$\lambda=\sqrt{1+\varepsilon}$:  for pointwise almost nonexpansiveness the 
inequality \eqref{e:epsqnonexp} must hold for all 
pairs $x^+\in Tx$ and $y^+\in Ty$, while for calmness the same inequality would hold only for points $x^+\in Tx$ and 
their {\em projections} onto $Ty$.  
\item See \cite[Example 2.2]{LukNguTam17} for concrete examples. 
\end{enumerate}
\end{remark}

\begin{proposition}[characterizations of pointwise averaged mappings]\label{t:average char}\cite[Proposition 2.1]{LukNguTam17}
   Let $\mmap{T}{\Ebb}{\Ebb}$, $\Omega\subset\Ebb$ and let $\alpha\in (0,1)$.  The following are equivalent.
\begin{enumerate}[(i)]
   \item\label{t:average char i} $T$ is pointwise almost averaged  on
$\Omega$ at $y\in \Omega$ with averaging constant $\alpha$ and
violation $\varepsilon$.
   \item\label{t:average char ii} $\paren{1-\frac1\alpha}\Id +
\frac1\alpha T$ is pointwise almost nonexpansive  on $\Omega$ at $y$
with violation
$\varepsilon/\alpha$.
   \item\label{t:average char iii}
\begin{equation}\label{e:average char iii}
  \begin{aligned}
\norm{x^+-  y^+}^2 \leq
(1+\varepsilon)\norm{x- y}^2 &- \frac{1-\alpha}{\alpha}\norm{\paren{x-x^+}-\paren{ y- y^+}}^2    \\
\forall x\in \Omega\quad \forall x^+ &\in Tx\quad \forall  y^+\in Ty.
\end{aligned}
\end{equation}
\end{enumerate}
As a consequence, if $T$ is pointwise almost averaged at $y$ with any averaging constant
$\alpha\in (0,1)$ and violation $\varepsilon$ on $\Omega$,
then $T$ is pointwise almost nonexpansive at $y$ with
violation at most $\varepsilon$ on $\Omega$.
\end{proposition}

\begin{remark}\label{Sin_of_Non}
Pointwise almost averaged mappings are single-valued on the set of fixed points \cite[Proposition 2.2]{LukNguTam17}.
If the mapping is actually nonexpansive (that is, almost nonexpansive with violation zero) on $\Omega$, then it must be
single-valued on $\Omega$. When this happens, we simply write $x^+=Tx$ instead of $x^+ \in Tx$.
\end{remark}

It was proved in \cite[Theorem 5.12]{BauschkeCombettes11} that if $(x_k)_{k\in\Nbb}$ is Fej\'er monotone
with respect to a nonempty closed convex subset $\Omega$ and inequality \eqref{LinCon_Seq+}
holds true with $\Omega$ in place of $S$, then $(x_k)_{k\in\Nbb}$ converges R-linearly to a point in
$\Omega$ with rate at most $c$.
The following statement aligns with this fact.

\begin{propn}[linear monotonicity and almost
averagedness imply R-linear convergence]\label{t:Com_LinCon}
Let $\mmap{T}{\Ebb}{\Ebb}$ and $x_{k+1}\in Tx_k\subset\Lambda\subset\Ebb$
 for all $k\in\Nbb$.
Suppose that $\Fix T\cap \Lambda$ is closed and nonempty and that $T$
is pointwise almost averaged at all points on $\left(\Fix T + d_0\Ball\right) \cap \Lambda$,
where $d_0:=\dist(x_0,\Fix T\cap \Lambda)$, that is,  $T$ is almost averaged on 
$\left(\Fix T + d_0\Ball\right) \cap \Lambda$.
If the sequence $(x_k)_{k\in\Nbb}$ is linearly monotone with respect to $\Fix T\cap \Lambda$ with constant
$c\in [0,1)$, then $(x_k)_{k\in\Nbb}$ converges R-linearly
to some point $\xtilde\in \Fix T\cap \Lambda$ with rate $c$.
\end{propn}

\begin{proof}
We use the characterization formulated in Proposition \ref{t:average char}\ref{t:average char iii} of the
almost averagedness of $T$ with averaging constant $\alpha$ and violation $\varepsilon$.
Combining \eqref{LinCon_Seq+} with $\Fix T\cap \Lambda$ in place of $S$ and \eqref{e:average char iii}
(with averaging constant $\alpha$ and violation $\varepsilon$)
implies by induction that for every $k\in\Nbb$,
\begin{align*}
\sqrt{\frac{1-\alpha}{\alpha(1+\varepsilon)}}
\|x_{k+1}-x_k\| \le \norm{x_k-\bx_k} =
\dist(x_{k},\Fix T\cap \Lambda)\le d_0 c^k,
\end{align*}
where $\bx_k$ is any point in $P_{\paren{\Fix T\cap \Lambda}}x_k$.
Hence, for any natural numbers $k$ and $p$ with $k<p$, we have
\begin{equation}\label{DR.3}
    \begin{aligned}
&\|x_{p}-x_{k}\| \,\le\, \sum_{i=k}^{p-1}\|x_{i+1}-x_{i}\| \,\le\,
\sqrt{\frac{\alpha(1+\varepsilon)}{1-\alpha}}\sum_{i=k}^{p-1}d_0 c^{i}\\
\le\;\; & d_0\sqrt{\frac{\alpha(1+\varepsilon)}{1-\alpha}}
(1+c+\ldots+c^{p-k-1})c^{k} \,\le\, \sqrt{\frac{\alpha(1+\varepsilon)}{1-\alpha}}\frac{d_0}{1-c}c^{k}.
    \end{aligned}
\end{equation}

This implies that $(x_k)_{k\in\Nbb}$ is a Cauchy sequence and therefore convergent to some point $\xtilde$.\\
We claim that $\xtilde\in \Fix T\cap \Lambda$.
Indeed, let us define
\[
\delta:=\frac{d_0}{1-c}\max\left\{\sqrt{\frac{\alpha(1+\varepsilon)}{1-\alpha}},1\right\}.
\]
The sequence $(\bx_k)_{k\in\Nbb}$ is contained in the bounded set 
$\Fix T \cap \Lambda \cap \Ball_{\delta}(x_0)$ since
\begin{align*}
\|\bx_k-x_0\| &\le \|\bx_k-x_k\|+\|x_k-x_0\|
\\
&\le d_0c^k + \sqrt{\frac{\alpha(1+\varepsilon)}{1-\alpha}}\sum_{i=0}^{k-1}d_0 c^{i}
\\
& \le \max\left\{\sqrt{\frac{\alpha(1+\varepsilon)}{1-\alpha}},1\right\}d_0\sum_{i=0}^{k}c^{i} < \delta.
\end{align*}
Hence it has a
subsequence $(\bx_{k_n})_{n\in\Nbb}$ converging to
some $\xtilde^*\in \Fix T\cap \Lambda$ as $n\to\infty$. Since the corresponding subsequence
$(x_{k_n})_{n\in\Nbb}$ converges to $\xtilde$ and
\[
\|\bx_{k_n}-x_{k_n}\|=\dist(x_{k_n},\Fix T\cap \Lambda)\le d_0 c^{k_n}\to 0
\]
as $n\to \infty$,
we deduce that $\xtilde=\xtilde^*\in \Fix T\cap \Lambda$.\\
Letting $p\to \infty$ in \eqref{DR.3} yields \eqref{LinCon_Seq}
with $\gamma= \sqrt{\frac{\alpha(1+\varepsilon)}{1-\alpha}}\frac{d_0}{1-c}$, which completes the proof.
\end{proof}

The converse implication 
of Proposition \ref{t:Com_LinCon} is not true in
general because condition \eqref{LinCon_Seq} in principle does
not require the distance $\norm{x_k-\xtilde}$ to strictly decrease after every iterate while
condition \eqref{LinCon_Seq+} does.

Almost nonexpansivity of $T$ and linear extendability of the iteration
are sufficient to guarantee that the
sequence converges R-linearly to a point in $\Fix T$.
Compare this to Proposition \ref{t:Com_LinCon_point} which, without the additional
assumption of 
almost nonexpansivity of $T$, only guarantees convergence to a point in $\Lambda$.

\begin{propn}[linear extendability and almost nonexpansivity imply R-linear convergence]\label{t:Com_LinCon_point2}
Let $\mmap{T}{\Ebb}{\Ebb}$ and $(x_k)_{k\in\Nbb}$ be a sequence generated by
$x_{k+1}\in Tx_k\subset\Lambda\subset\Ebb$ for all $k\in\Nbb$.
Suppose that $(x_k)_{k\in\Nbb}$ is linearly extendible on $\Lambda$ with some frequency $m\ge 1$ and rate $c\in [0,1)$
and that $T$ is almost nonexpansive on $\Lambda\cap \Ball_{\gamma}(x_0)$,
where $\gamma$ is given by \eqref{gamma}.
 Then $(x_k)_{k\in\Nbb}$ converges R-linearly
to a point $\xtilde\in \Fix T\cap \Lambda$ with rate $c$.
\end{propn}

\begin{proof}
By Proposition \ref{t:Com_LinCon_point} the sequence $(x_k)_{k\in\Nbb}$ converges R-linearly to a
point $\xtilde\in \Lambda$ with rate $c$.
It remains to check that $\xtilde\in \Fix T$.
Suppose to the contrary that there is $\xtilde^+ \in T\xtilde$ with $\rho:=\|\xtilde^+-\xtilde\|>0$.
Since $T$ is almost nonexpansive  on $\Lambda\cap \Ball_{\gamma}(x_0)$, there is a violation $\varepsilon>0$
such that 
\[
\|\xtilde^+-x_{k+1}\| \le \sqrt{1+\varepsilon}\|\xtilde-x_k\|\quad \forall k\in \Nbb.
\]
This leads to a contradiction since $\|\xtilde^+-x_{k+1}\| \to \rho>0$
while $\|\xtilde-x_k\|\to 0$ as $k\to \infty$.
As a result, $\xtilde^+=\xtilde\in \Fix T$ and the proof is complete.
\end{proof}

When the fixed point set restricted to $\Lambda$ is an isolated point, then
linear monotonicity of the sequence
is equivalent to Q-linear convergence.

\begin{propn}\label{t:Com_QLinCon}
Let $\xtilde$ be an isolated point of $\Fix T\cap\Lambda$, that is $\Ball_\delta(\xtilde)\cap \Fix T\cap\Lambda=\{\xtilde\}$
for $\delta>0$ small enough. 
Let $\mmap{T}{\Ebb}{\Ebb}$ be almost
nonexpansive on a neighborhood of $\xtilde$ relative to $\Lambda\subset\Ebb$.
Let $(x_k)_{k\in\Nbb}$ be a sequence generated by $x_{k+1}\in Tx_k\subset\Lambda$ for all $k\in\Nbb$
with $x_0\in\Lambda$
sufficiently close to $\xtilde$.
Then $(x_k)_{k\in\Nbb}$
is linearly monotone with respect to $\Fix T\cap \Lambda$ with rate smaller than 1 if and only if
 it is Q-linearly convergent to $\xtilde$.
\end{propn}
\begin{proof}
Since $\xtilde$ is an isolated piont of $\Fix T\cap \Lambda$ and 
$T$ is almost nonexpansive on a neighborhood of $\xtilde$ relative to $\Lambda$, there is 
a $\delta>0$ small enough that
$\Fix T\cap \Lambda\cap \Ball_\delta(\xtilde) = \{\xtilde\}$ and $T$ is
almost nonexpansive on
$\Ball_{\delta'}(\xtilde)\cap\Lambda$ with
violation $\varepsilon$, where $\delta'=\frac{\delta}{2\sqrt{1+\varepsilon}}$.
Let $\rho \in(0,\delta')$.
Then by 
almost nonexpansivity of $T$ on $\Ball_{\rho}(\xtilde)\cap\Lambda$ we have that
\begin{eqnarray*}
\|x^+-\xtilde\| &\le& \sqrt{1+\varepsilon}\|x-\xtilde\| \le
\rho\sqrt{1+\varepsilon}\\
&<& \frac{\delta}{2} \le
\frac{1}{2}\dist\left(\xtilde,\paren{\Fix T\cap \Lambda}
\setminus \{\xtilde\}\right)\quad \forall x\in \Ball_{\rho}(\xtilde)\cap\Lambda, \forall 
x^+\in Tx.
\end{eqnarray*}
This implies that
\begin{align}\label{mon_lin:equ}
\dist(x^+,\Fix T\cap \Lambda) = \|x^+-\xtilde\|\quad \forall x\in \Ball_{\rho}(\xtilde)\cap\Lambda, \forall 
x^+\in Tx.
\end{align}
Hence for any sequence $(x_k)_{k\in\Nbb}$ as described in Proposition \ref{t:Com_QLinCon} with
$x_0\in \Lambda \cap \Ball_{\rho}(\xtilde)$, the equivalence of linear monotonicity
of $(x_k)_{k\in\Nbb}$ relative to $\Fix T\cap \Lambda$ with rate smaller than 1 and Q-linear
convergence to $\xtilde$ follows from equality \eqref{mon_lin:equ} and the
definitions because each of these properties of $(x_k)_{k\in\Nbb}$ alternatively combined
with \eqref{mon_lin:equ} ensures inductively that the whole sequence
$(x_k)_{k\in\Nbb}$ lies in $\Ball_{\rho}(\xtilde)$.
\end{proof}

It is mainly due to the above proposition that we include the extra technical
overhead of making the above statements always relative to some
subset $\Lambda$.  It is not uncommon to have $\Fix T$ a singleton
relative to $\Lambda$, but not on the whole space $\Ebb$. For an example of
this, see the analysis of the Douglas-Rachford fixed point iteration in
\cite{AspChaLuk16}.

\subsection{Metric subregularity and linear convergence}

The following concept of \emph{metric regularity on a set} characterizes the
stability of mappings at points in their image and has played a central role,
implicitly and explicitly, in our convergence analysis of fixed point iterations
\cite{HesLuk13, AspChaLuk16, LukNguTam17}.  We show in this section that, indeed, metric subregularity is
{\em necessary} to achieve linear convergence.
The following definition is a specialized (linear) variant of \cite[Definition 2.5]{LukNguTam17}
which is a combination with slight modification of
those formulated in \cite[Definition 2.1 (b)] {Ioffe11} and \cite[Definition 1 (b)]{Ioffe13} so that the property is relative to
relevant sets for iterative methods. Our terminology also follows \cite{DontRock09}.

\begin{defn}[metric regularity on a set]\label{d:(str)metric (sub)reg}
$~$ Let $\mmap{\Phi}{\Ebb}{\Ybb}$ be a set-valued mapping between Euclidean 
spaces and let $U\subset \Ebb$ and $V\subset \Ybb$.
The mapping $\Phi$ is called \emph{metrically regular relative to
$\Lambda\subset\Ebb$} on $U$ for $V$ with constant $\kappa$ if
\begin{equation}\label{e:metricregularity}
\dist\paren{x, \Phi^{-1}(y)\cap \Lambda}\leq \kappa\dist\paren{y, \Phi(x)}
\end{equation}
holds for all $x\in U\cap \Lambda$ and $y\in V$.

When $V=\{y\}$ consists of a single point one says that $\Phi$ is
\emph{metrically subregular}
with constant $\kappa$ on $U$ for $y$ relative to
$\Lambda\subset\Ebb$.

When $\Lambda=\Ebb$, the quantifier ``relative to'' is dropped.
\end{defn}
\begin{remark}
The conventional concept of \emph{metric regularity} at a
point $\xbar\in \Ebb$ for $\ybar\in \Phi(\xbar)$ corresponds to setting
$\Lambda=\Ebb$, and taking $U$ and $V$ to be neighborhoods of $\xbar$ and $\ybar$
(as opposed simply to subsets including these points) respectively.
Similarly, the conventional \emph{metric subregularity} \cite{DontRock09}
 and \emph{metric hemi/semiregularity}
\cite{KrugerNguyen15,ArtMor11,Kruger09} at $\xbar$ for $\ybar$ correspond to setting
$\Lambda=\Ebb$, and respectively either taking $U$ to be a neighborhood of $\xbar$ and $V=\{\ybar\}$, or taking
$U=\{\xbar\}$ and taking $V$ to be a
neighborhood of $\ybar$.  This notion can and has been generalized even more.
The more general notion of
metric subregularity studied by Ioffe \cite{Ioffe11, Ioffe13} for instance, together with
$\mu$-monotonicity, would be needed for the study of nonlinear convergence.  These more general notions of
metric subregularity are the most suitable vehicles to parallel properties like the
Kurdyka-{\L}ojasiewicz (KL) property for functions. In fact, for differentiable functions,
{\em metric regularity} of the gradient is equivalent to the KL property \cite[Corollary 4 and Remark 5]{BolDan2010},
though from our point of view, metric subregularity is the more general property.
\end{remark}

The following convergence criterion is fundamental.

\begin{thm}[linear convergence with metric subregularity]\label{t:metric subreg convergence}
Let $\mmap{T}{\Ebb}{\Ebb}$, let $\Lambda\subset\Ebb$ with
$\Fix T\cap\Lambda$ closed and nonempty.
Suppose that, for some fixed $\delta>0$,  $T$ is pointwise almost averaged at all points 
$\xbar\in \Fix T\cap \Lambda$ with averaging constant
$\alpha$ and violation $\varepsilon$ on
$\paren{\Fix T+\Bbb_{\delta}}\cap\Lambda$, and that the mapping $\Phi\equiv T-\Id$ is metrically subregular on
$(\Fix T+\Bbb_{\delta})\setminus \Fix T$ for $0$ relative to $\Lambda$ with constant $\kappa>0$.
Then it holds
\begin{equation}\label{e:key}
\begin{aligned}
\dist\paren{x^+,\Fix T\cap\Lambda} \le c\dist\paren{x,\Fix T\cap\Lambda}\\
\forall x\in \paren{\Fix T+\Bbb_{\delta}}\cap \Lambda\quad \forall x^+ \in Tx,
\end{aligned}
\end{equation}
where
\begin{equation}\label{e:msr rate const}
c:=\sqrt{1+\varepsilon-\frac{1-\alpha}{\kappa^2\alpha}}.
\end{equation}
In particular, when $c<1$,
every sequence $(x_k)_{k\in\Nbb}$ generated by $x_{k+1}\in Tx_k\subset\Lambda$ with initial point
in $\paren{\Fix T+\Bbb_{\delta}}\cap\Lambda$ converges R-linearly to some point in $\Fix T\cap\Lambda$ with
rate $c$.  If $\Fix T\cap\Lambda$ is a singleton, then the convergence is Q-linear.
\end{thm}

\begin{proof}
The inequality \eqref{e:key} is the content of \cite[Corollary 2.3]{LukNguTam17}.  Since this is easy to obtain,
we reproduce the proof here for convenience.
Choose any $x\in (\Fix T+\Bbb_{\delta})\cap \Lambda$ and   select any $x^+\in Tx$.
Metric subregularity of $\Phi$ on $(\Fix T+\Bbb_{\delta})\setminus \Fix T$ for $0$ relative to $\Lambda$
with constant $\kappa>0$ means that
\[
   \dist\paren{x, \Phi^{-1}(0)\cap \Lambda}\leq \kappa \dist(0,\Phi(x)).
\]
Since $\Phi^{-1}(0)\cap \Lambda=\Fix T\cap \Lambda$, this then implies that
\[
   \tfrac{1}{\kappa^2}\dist^2\paren{x, \Fix T\cap \Lambda}\leq \|x^+-x\|^2.
\]
Note that $T$ is a single-valued mapping on $\Fix T\cap\Lambda$
since $T$ is almost averaged  -- hence almost nonexpansive -- on $\paren{\Fix T+\Bbb_{\delta}}\cap\Lambda$
\cite[Proposition 2.2]{LukNguTam17}, so we can unambiguously write $\xbar= T\xbar$ for $\xbar\in P_{\paren{\Fix T\cap\Lambda}}x$
and rewrite the inequality as
\[
   \tfrac{1}{\kappa^2}\|x-\xbar\|^2\leq \|x^+-x\|^2 \quad \forall x\in (\Fix T+\Bbb_{\delta})\setminus \Fix T.
\]
This inequality, together with the almost averaging property and its characterization 
Proposition~\ref{t:average char}\ref{t:average char iii}
yield
\begin{equation}
   \norm{x^+-\xbar}^2\leq\paren{1+\varepsilon-\frac{1-\alpha}{\alpha\kappa^2}}\norm{x-\xbar}^2.
\end{equation}%
Note in particular that  $0\leq 1+\varepsilon-\frac{1-\alpha}{\alpha\kappa^2}$.
Since $x$ is any point in $(\Fix T+\Bbb_{\delta})\cap \Lambda$ this proves \eqref{e:key} with $c$
given by \eqref{e:msr rate const}.

 For $c<1$, it follows by
definition that such a sequence $(x_k)_{k\in\Nbb}$ is linearly monotone with respect to
$\Fix T\cap\Lambda$ with rate $c$. A combination of  Propositions
\ref{t:Com_LinCon} and \ref{t:Com_QLinCon} then shows that the sequence is linearly convergent
to a point in $\Fix T\cap\Lambda$,
R-linearly in general, and Q-linearly if $\Fix T\cap\Lambda$ is a singleton.
\end{proof}

When ${\delta}=\infty$, Theorem \ref{t:metric subreg convergence} provides a criterion for global
linear convergence of abstract fixed point
iterations. The next result shows that metric subregularity is
in fact {\em necessary} for linearly monotone
iterations, without any assumptions about the
averaging properties of $T$, almost or otherwise.

\begin{thm}[necessity of metric subregularity]\label{t:Nec_for_msr}
Let $\mmap{T}{\Ebb}{\Ebb}$ with $\Fix T$ nonempty, fix $\Lambda\subset\Ebb$ so that 
$\Fix T\cap\Lambda$ is closed and nonempty, and let $\Omega\subset \Lambda$.
If for each $x_0\in \Omega$, every sequence $(x_k)_{k\in\Nbb}$ generated by
$x_{k+1}\in Tx_k\subset\Lambda$ is linearly monotone
with respect to $\Fix T\cap\Lambda$ with constant $c\in [0,1)$, then the mapping $\Phi\equiv T-\Id$ is metrically
subregular on $\Omega$ for $0$ relative to $\Lambda$ with constant $\kappa \le \frac{1}{1-c}$.
\end{thm}

\begin{proof} Since
every sequence $(x_k)_{k\in\Nbb}$ generated by $x_{k+1}\in Tx_k\subset\Lambda$ starting
in $\Omega$ is linearly
monotone with respect to $\Fix T\cap\Lambda$ with rate $c$, the inequality \eqref{LinCon_Seq+}
with $\Fix T\cap\Lambda$ in place of $S$ holds true.
This together with the triangle inequality
implies that for every $k\in \Nbb$,
\begin{equation}\label{Nec_MetSub}
  \begin{aligned}
\norm{x^+-x_k} &\ge \dist\paren{x_k,\Fix T\cap\Lambda} - \dist\paren{x^+,\Fix T\cap\Lambda}
\\
&\ge (1-c)\dist\paren{x_k,\Fix T\cap\Lambda} \quad \forall x^+\in Tx_k.
\end{aligned}
\end{equation}
On the other hand, we have from definition of $\Phi$ that
\begin{equation}\label{Phi=T}
\begin{aligned}
\Phi^{-1}(0) &= \Fix T,
\\
\dist\paren{0,\Phi(x_k)} &= \inf_{x^+\in Tx_k}\norm{x^+-x_k} \quad \forall k\in \Nbb.
\end{aligned}
\end{equation}
Combining \eqref{Nec_MetSub} and \eqref{Phi=T} yields
\[
\dist\paren{x_k,\Phi^{-1}(0)\cap\Lambda} \le \frac{1}{1-c}\dist\paren{0,\Phi(x_k)} \quad\forall k\in\Nbb.
\]
Consequently, $\frac{1}{1-c}$ is {\em a} constant of metric subregularity of $\Phi$ on $\Omega$ for $0$ 
(not necessarily the smallest such constant) as claimed.
\end{proof}

\begin{cor}[necessary conditions for linear convergence]\label{t:nec_suff}
Let $\mmap{T}{\Ebb}{\Ebb}$ with $\Fix T$ nonempty.  For some $\delta>0$, let $T$ 
be almost averaged 
on $\paren{\Fix T+\Bbb_{\delta}}\cap \Lambda$.
If, for each
$x_0\in \paren{\paren{\Fix T+\Bbb_{\delta}}\cap\Lambda}\setminus \Fix T$, every sequence
$(x_k)_{k\in\Nbb}$ generated by $x_{k+1}\in Tx_k\subset\Lambda$
is linearly monotone with respect to $\Fix T\cap \Lambda$ with constant $c \in [0,1)$,
then all such sequences converge R-linearly with rate $c$ to some point in $\Fix T\cap \Lambda$
and $\Phi\equiv T-\Id$ is metrically subregular on $\paren{\Fix T+\Bbb_{\delta}}\setminus \Fix T$
for $0$ relative to $\Lambda$ with constant $\kappa\leq \frac{1}{1-c}$.
\end{cor}
\begin{proof}
   This is an immediate consequence of Proposition \ref{t:Com_LinCon} and Theorem \ref{t:Nec_for_msr}.
\end{proof}

\section{Nonconvex alternating projections}\label{s:AP_for_NonConvex}
For $x_0\in \Ebb$ given, the {\em alternating projections (AP) iteration}
for two closed
subsets $A, B\subset \Ebb$ is given by
\begin{equation}\label{e:ap}
x_{k+1}\in T_{AP}x_{k} \equiv P_AP_Bx_{k} \quad \forall k\in \Nbb.
\end{equation}
For convenience, we associate $(x_k)_{k\in\Nbb}$ with the sequence $(b_k)_{k\in\Nbb}$ on $B$ such that
$b_k\in P_Bx_{k}$ and $x_{k+1}\in P_Ab_k$ for all $k\in \Nbb$.
In the sequel, we frequently use the joining sequence $(z_k)_{k\in\Nbb}$ given by
\begin{align}\label{z_k}
z_{2k} = x_k \mbox{ and } z_{2k+1} = b_k \quad \forall k\in \Nbb.
\end{align}
We will always assume, without loss of generality, that $x_0\in A$.

It is well known that every alternating projections iteration for two convex intersecting sets globally converges
R-linearly to a feasibility solution if the collection of sets is what we call \emph{subtransversal} \cite{BauBorSIREV96}.
The later property and its at-point version is a specialization of metric subregularity to the context of set feasibility.

\subsection{Elemental regularity and subtransversality}
Convexity of the underlying sets has long been the standing assumption in analysis of projection algorithms.
The next definition characterizing regularity of nonconvex sets first appeared in \cite[Definition 5]{KruLukNgu16}
and encapsulates many of the regularity notions appearing elsewhere in the literature including
H\"older regularity \cite[Definition 2]{NolRon16}, relative ($\varepsilon,\delta$)-subregularity
\cite[Definition 2.9]{HesLuk13}, restricted $(\varepsilon,\delta)$-regularity \cite[Definition 8.1]{BauLukePhanWang13b},
Clarke regularity \cite[Definition 6.4]{VA}, super-regularity \cite[Definition 4.3]{LewLukMal09},
prox-regularity \cite[Definition 1.1]{PolRocThi00}, and of course convexity.   The connection of elemental regularity
of sets to the pointwise almost averaging property of their projectors is discussed later.
\begin{definition}[elemental regularity of sets]\label{d:set regularity}
    Let $A\subset\Ebb$ be nonempty and let $\paren{ a,\vbar}\in\gph\paren{\ncone{A}}$.
\begin{enumerate}[(i)]
   \item\label{d:geom subreg} The set $A$ is said to be {\em elementally subregular relative to
$ S\subset A$ at $\xbar\in A$ for $\paren{ a,\vbar}$  with constant $\varepsilon$}
 if there exists a neighborhood $U$ of $\xbar$ such that
   \begin{equation}\label{e:geom subreg}
     \ip{\vbar }{x -  a}\leq \varepsilon\norm{\vbar}\norm{x -  a},\quad \forall x\in  S \cap U.
   \end{equation}

   \item\label{d:uni geom subreg} The set $A$ is said to be {\em uniformly elementally subregular
relative to $ S\subset A$ at $\xbar$}  for $( a,\vbar)$ if, for any $\varepsilon>0$,
there exists a neighborhood $U$ (depending on $\varepsilon$) of $\xbar$  such that \eqref{e:geom subreg} holds.

 \item\label{d:geom reg} The set $ A$ is said to be {\em elementally regular at $\xbar$ for
$\paren{ a,\vbar}$ with constant $\varepsilon$} if there exists a neighborhood $V$ of $\vbar$ such
that, for each $v\in\ncone{ A}( a)\cap V$, the set $ A$ is elementally subregular relative to
$ S= A$ at $\xbar$ for $\paren{ a,v}$ with constant $\varepsilon$.
 \item\label{iv:uniform element regular} The set $ A$ is said to be {\em uniformly elementally regular
at $\xbar$ for $\paren{ a,\vbar}$} if, for any $\varepsilon>0$, the set $ A$ is elementally regular at 
$\xbar$ for $\paren{ a,\vbar}$ with constant $\varepsilon$.
\end{enumerate}
\end{definition}

The reference points $\xbar$ and $a$ in Definition \ref{d:set regularity},
need not be in $S$ and  $U$, respectively,
although these are the main cases of interest for us.
The properties are trivial for any constant $\varepsilon\ge 1$, so the only
case of interest is elemental (sub)regularity with constant $\varepsilon<1$.

\begin{proposition}[Proposition 4(vii) of \cite{KruLukNgu16}]\label{t:set regularity equivalences}
Let $A$ be closed and nonempty.
If $ A$ is convex, then it is
elementally regular at all points $x \in A$ for all
$(a,v)\in \gph{\ncone{A}}$ with constant $\varepsilon=0$
and the neighborhood $\Ebb$ for both $x$ and $a$.
\end{proposition}

The next result shows the implications of elemental regularity of sets for regularity of the
corresponding projectors.

\begin{thm}[projectors onto elementally subregular sets, Theorem 3.1 of \cite{LukNguTam17}]
\label{t:subreg proj-ref}
Let ${A}\subset\Ebb$ be nonempty closed, and let $U$ be a neighborhood of $\xbar\in {A}$.
Let $ S\subset{A}\cap U$ and $ S'\equiv P_{{A}}^{-1}( S)\cap U$.
If ${A}$ is elementally sub\-reg\-ul\-ar at $\xbar$
relative to $ S'$ for each
\[
(a,v)\in V\equiv\set{(z, w)\in\gph\pncone{{A}}}{ z+w\in U\und z\in P_{A}(z+w)}
\]
 with constant $\varepsilon$ on the neighborhood $U$, then the following hold.
\begin{enumerate}[(i)]
\item\label{t:subreg proj-ref1}
The projector $P_{A}$ is pointwise almost nonexpansive at each $y\in  S$ on
$U$
with violation $\varepsilon'\equiv 2\varepsilon+\varepsilon^2$. That is
\[
\norm{x -y}\leq \sqrt{1+\varepsilon'}\norm{x'-y}\quad \forall y\in  S \quad \forall x'\in U\quad \forall x\in P_{A} x'.
\]
\item\label{t:subreg proj-ref2}
The projector $P_{A}$ is pointwise almost firmly nonexpansive at each $y\in  S$ on $U$
with violation $\varepsilon'_2\equiv2\varepsilon+2\varepsilon^2$. That is
\[
  \begin{aligned}
\norm{x -y}^2+\norm{x'-x }^2 &\leq (1+\varepsilon'_2)\norm{x'-y}^2\\
\forall y\in  S\quad \forall x' &\in U\quad \forall x\in P_{A} x'.
\end{aligned}
\]
\end{enumerate}
\end{thm}

In addition to the pointwise almost averaging property, metric subregularity plays a central role in the
general theory of Section \ref{s:lmfp}.  In the context of set feasibility, this is translated to what we call
subtransversality below.
What we present as the definition of subtransversality is
the simplified version of \cite[Definition 3.2(i)]{LukNguTam17}.

\begin{defn}[subtransversality]\label{d:MAINi}
Let $A$ and $B$ be closed subsets of $\Ebb$, let $\Ebb^2$ be endowed with some norm and let $\Gamma\subset \Ebb^2$.
The collection of sets $\{A,B\}$ is said to be subtransversal at $\ubar = (\xbar_1,\xbar_2) \in A\times B$ for $\wbar =
(\ybar_1,\ybar_2) \in \paren{P_A-\Id}(\xbar_2)\times \paren{P_B-\Id}(\xbar_1)$ relative to $\Gamma$ if there exist numbers
$\delta>0$ and $\kappa\ge0$
such that
\begin{gather*}\label{e:T1-1}
\dist\left(u,\paren{\paren{P_B-\Id}^{-1}(\ybar_2)\times \paren{P_A-\Id}^{-1}(\ybar_1)}\cap\Gamma\right)
\le \kappa\dist\paren{\wbar,\paren{P_A-\Id}(x_2)\times \paren{P_B-\Id}(x_1)}
\end{gather*}
holds true for all $u = (x_1,x_2) \in \B_{\delta}(\ubar )\cap \Gamma.$

When $\Gamma=\Ebb^2$, the quantifier ``relative to'' is dropped.
\end{defn}
The reference point $\ubar$ in Definition \ref{d:MAINi} need not be in $\Gamma$,
although this is the only case of interest for us.
The following characterization of subtransversality at common points
 will play a fundamental role in our subsequent analysis.

\begin{propn}[subtransversality at common points, Proposition 3.3 of \cite{LukNguTam17}]\label{t:metric characterization}
Let $A$ and $B$ be closed subsets of $\Ebb$.  Let $\Ebb^2$ be endowed with $2$-norm
\[
\norm{(x_1,x_2)}_2= \paren{\norm{x_1}^2_{\Ebb}+\norm{x_2}^2_{\Ebb}}^{1/2}\quad \forall\; (x_1,x_2)\in \Ebb^2.
\]
 The collection of sets $\{A,B\}$ is subtransversal relative to
\begin{equation}\label{Gamma_set}
\Gamma\equiv\set{u=(x,x)\in\Ebb^2}{x\in \Lambda\subset \Ebb}
\end{equation}
at $\ubar = \paren{\xbar,\xbar}$ with $\xbar\in A\cap B$ for
$\ybar=0_{\Ebb^2}$ if and only if there exist numbers $\delta>0$ and $\kappa\ge0$ such that
\begin{equation}\label{eq:locallinear+}
\dist\paren{x,A\cap B\cap \Lambda}\leq \kappa \max\paren{\dist\paren{x,A},\dist\paren{x,B}}
\quad\forall x\in \Ball_{\delta}(\xbar)\cap \Lambda.
\end{equation}
\end{propn}

The relative set $\Gamma\subset \Ebb^2$ given by \eqref{Gamma_set} which makes the notion of
subtransversality consistent in the product
space can clearly be identified with the set $\Lambda\subset \Ebb$.
We will therefore more often than not use the terms ``relative to $\Lambda$'' instead of ``relative to $\Gamma$''
and ``at $\xbar$'' instead of ``at $(\xbar,\xbar)$ for $0_{\Ebb^2}$'' when discussing subtransversality at
common points where the product-space structure is no longer needed.

\begin{remark}\label{sr<inf}
It follows from  Proposition \ref{t:metric characterization} that the exact lower bound of all numbers $\kappa$ 
such that condition
\eqref{eq:locallinear+} is satisfied,
denoted $\sr$, characterizes the subtransversality of the collection of sets at common points. More specifically, the collection of sets
$\{A,B\}$ is
subtransversal at
$\xbar$ if and only if $\sr<+\infty$.
\end{remark}

The property \eqref{eq:locallinear+} with $\Lambda=\Ebb$
 has been around for decades under
the names of (local) \emph{linear regularity}, \emph{metric regularity}, \emph{linear coherence},
\emph{metric inequality}, and \emph{subtransversality}; cf. \cite{BauBorSVA,
BauBorSIREV96,Ioffe89,Ioffe00,HesLuk13,
NgaiThera01,PenotCWD,ZheNg08,DruIofLew15,Kruger06}.
We refer the reader to the recent articles \cite{KruLukNgu16,KruLukNgu17} in which a number of necessary
and/or sufficient characterizations of subtransversality are discussed.
The next characterization of subtransversality, which is the
relativized version of \cite[Theorem 1(iii)]{KruLukNgu16}, will play a key role in proving the necessary
condition results in Sections \ref{subs:N&S} and \ref{s:AP_with_Convex}.
This characterization is actually implied in the proof of \cite[Theorem 6.2]{DruIofLew15} where
 the property called \emph{intrinsic transversality} \cite[Definition 3.1]{DruIofLew15} was shown
to imply subtransversality.

\begin{proposition}[characterization of subtransversality at common points]\label{t:MAINid}
The collection of sets $\{A,B\}$ is subtransversal at $\xbar \in A\cap B$ relative to $\Lambda$ if and only if there exist numbers
$\delta>0$ and $\kappa\ge 0$
such that
\begin{equation}\label{e:P3-1}
\dist(x,A\cap B\cap \Lambda)\le \kappa\dist(x,B)\quad \forall x\in A\cap\B_{\delta}(\xbar )\cap\Lambda.
\end{equation}
Moreover,
\begin{equation}\label{P3-2}
\srr \le \sr\le 1+2\srr,
\end{equation}
where $\srr$ is the exact lower bound of all numbers $\kappa$ such that condition (\ref{e:P3-1}) is satisfied.
\end{proposition}

\begin{remark}
In light of Remark \ref{sr<inf} and the two inequalities in \eqref{P3-2}, a collection of sets
$\{A,B\}$ is subtransversal at $\xbar\in A\cap B$ if and only if $\srr<+\infty$.
For the simplicity in terms of presentation, in the sequel, we will frequently use this fact without repeating the argument.
\end{remark}
Both inequalities in \eqref{P3-2} can be strict as shown in the following example.
\begin{eg}
Let $A$ and $B$ be two lines in $\mathbb{R}^2$ forming an angle $\pi/3$ at the intersection point $\bx$.
One can easily check that
\[
\srr = 2/\sqrt{3} < 2 = \sr < 1+2\srr = 1+4/\sqrt{3}.
\]
\end{eg}

The connection of subtransversality to metric subregularity was presented for more general cyclic projections in
\cite[Proposition 3.4]{LukNguTam17}.  We present here the simplified version for two sets with possibly empty intersection.

\begin{propn}[metric subregularity
for alternating projections]\label{t:msr cp}
Let $A$ and $B$ be closed nonempty sets.
Let $\xbar_1\in \Fix T_{AP}$ and $\xbar_2\in P_B\xbar_1$ such that $\xbar_1\in P_A\xbar_2$ and let $\Gamma$ be
the affine subspace
\[
\Gamma := \left\{(x,x+\xbar_2-\xbar_1) : x\in \Ebb\right\}\subset \Ebb^2.
\]
Define $\Phi\equiv T_{AP}-\Id$.
Suppose the following hold:
\begin{enumerate}[(a)]
\item\label{t:msr cp i} the collection of sets $\{A, B\}$ is subtransversal at $\ubar = (\xbar_1,\xbar_2)$ for $\ybar = (\xbar_1-\xbar_2,\xbar_2-\xbar_1)$
relative to $\Gamma$ with constant $\kappa$ and neighborhood $U$ of $\ubar$;
 \item\label{t:msr cp ii}
there exists a positive constant $\sigma$ such that
\begin{align}\label{constraint}
\dist\paren{\ybar,(P_A-\Id)(x_2)\times (P_B-\Id)(x_1)} \leq \sigma \dist\paren{0,\Phi(x)}\quad \forall x=(x_1,x_2)\in U\cap \Gamma.
\end{align}
\end{enumerate}
Then the mapping $\Phi$ is metrically subregular on $U$ for $0$ relative to
$\Gamma$ with constant $\kappa\sigma$.
\end{propn}

\subsection{Necessary and sufficient conditions for local linear convergence}\label{subs:N&S}

It was established in \cite[Corollary 13(a)]{HesLuk13} that {\em local linear regularity} of the collection of
sets (with a reasonably good quantitative constant as always for convergence analysis of
nonconvex alternating projections) is sufficient for linear monotonicity of the method for
$(\varepsilon,\delta)$-subregular sets.  This result is updated here in light of more recent results.

\begin{propn}[convergence of alternating projections with nonempty intersection]\label{t:cp ncvx}
Let $S$ be a nonempty subset of $A\cap B$.
Let $U$ be a neighborhood of $S$ such that
\begin{equation}\label{e:Uj}
P_{A}(U)\subseteq U \und P_B(U)\subseteq U.
\end{equation}
Let $\Lambda$ be an affine subspace
 containing $S$
 such that $\mmap{T_{AP}}{\Lambda}{\Lambda}$.
Define $\Phi\equiv T_{AP}-\Id$.
Let the sets $A$ and $B$ be elementally subregular at all $\xhat\in S$ relative to $\Lambda$ respectively for each
\begin{eqnarray*}
(a, v_{A})\in V_{A}&\equiv&\set{(z, w)\in\gph\pncone{{A}}}{ z+w\in U\und z\in P_{{A}}(z+w)} \\
(b, v_{B})\in V_{B}&\equiv&\set{(z, w)\in\gph\pncone{{B}}}{ z+w\in U\und z\in P_{{B}}(z+w)} \\
\end{eqnarray*}
with respective constants $\varepsilon_A, \varepsilon_B\in [0,1)$ on the neighborhood $U$.
Suppose that the following hold:
\begin{enumerate}[(a)]
\item\label{t:cp ncvx ii} for each $\xhat\in S$, the collection of sets
$\{A, B\}$ is subtransversal at
$\xhat$ relative to
$\Lambda$ with constant $\kappa$ on the neighborhood $U$;
 \item\label{t:cp ncvx iii}
there exists a positive constant $\sigma$ such that condition \eqref{constraint} holds true;
\item\label{t:cp ncvx iv} $\dist(x, S) \le \dist\paren{x, A\cap B\cap \Lambda}$ for all $x\in U\cap \Lambda$;
\item $\varepsilontilde_A+\varepsilontilde_B+\varepsilontilde_A\varepsilontilde_B < \frac{1}{2(\kappa\sigma)^2}$, where
 $\varepsilontilde_A\equiv
  4\varepsilon_A\frac{1+\varepsilon_A}{\paren{1-\varepsilon_A}^2}$
   and $\varepsilontilde_B\equiv
  4\varepsilon_B\frac{1+\varepsilon_B}{\paren{1-\varepsilon_B}^2}$.
\end{enumerate}
Then every sequence $\paren{x_k}_{k\in \Nbb}$ generated by $x_{k+1}\in T_{AP} x_k$
seeded by any point $x_0\in A\cap U\cap\Lambda$ is linearly monotone with respect to $S$ with constant
\begin{equation*}
c\equiv\sqrt{1+\varepsilontilde_A+\varepsilontilde_B+\varepsilontilde_A\varepsilontilde_B - \frac{1}{2(\kappa\sigma)^2}
}\;\;\; <\; 1.
\end{equation*}
Consequently,
 $\dist\paren{x_k,S}\to 0$ at least Q-linearly with rate $c$.
\end{propn}
\begin{proof}
In light of Proposition \ref{t:msr cp}
and the definition of linear monotonicity,
 Proposition \ref{t:cp ncvx} is a specialization of \cite[Theorem 3.2]{LukNguTam17} to the case of two sets
with nonempty intersection.
\end{proof}

If $S=A\cap B\cap \Lambda$ in Proposition \ref{t:cp ncvx}, then assumption \ref{t:cp ncvx iv} can obviously be omitted.

The next theorem shows that the converse to Proposition \ref{t:cp ncvx} holds more generally without any
assumption on the elemental
regularity of the individual sets.
The proof of the next theorem uses the idea in the proof of \cite[Theorem 6.2]{DruIofLew15}.

\begin{thm}[subtransversality is necessary for linear monotonicity of subsequences]\label{NonCon_Nec1+}
Let $\Lambda$, $A$, and $B$ be closed subsets of $\Ebb$,
let  $\xbar\in S\subset A\cap B\cap\Lambda$,
and let $1\le n\in \mathbb{N}$ and $c\in [0,1)$ be fixed.
Suppose that for any sequence of
alternating projections $(x_k)_{k\in \Nbb}$
starting
in $\Lambda$ and
sufficiently close to $\xbar$, there exists a subsequence of the form $(x_{j+nk})_{k\in \Nbb}$ for
some $j\in \{0,1,\ldots,n-1\}$ that remains in $\Lambda$ and is linearly monotone with respect to $S$
with constant $c$.
Then the collection of sets $\{A,B\}$ is subtransversal at $\xbar$ relative to $\Lambda$
with constant $\srr\le \frac{2(2n^2-1-c(n-1))}{1-c}$.
\end{thm}
\begin{proof}
Let $\rho>0$ be so small that any alternating projections sequence $(x_k)_{k\in \Nbb}$ starting in
$\Ball_{\rho}(\xbar)\cap \Lambda$ has a subsequence $(x_{j+nk})_{k\in \Nbb}$ which is linearly
monotone with respect to $S$ with constant
 $c$.
Take any $x_0\in A\cap \Ball_{\rho}(\xbar)\cap \Lambda$.
Let us consider any alternating projections sequence $(x_k)_{k\in\Nbb}$ starting at $x_0$ and such a
subsequence $(x_{j+nk})_{k\in\Nbb}$.
On one hand,
\begin{equation}\label{n dist>dist}
  \begin{aligned}
2n \dist(x_j,B) &\ge \|x_j-x_{j+n}\|
\ge \dist(x_j,S)-\dist(x_{j+n},S)
\\
&\ge (1-c)\dist(x_j,S) \ge (1-c)\dist(x_j,A\cap B\cap \Lambda)
\\
&\ge (1-c)\left(\dist(x_0,A\cap B\cap \Lambda)-\|x_0-x_j\|\right).
\end{aligned}
\end{equation}
On the other hand,
\begin{equation}\label{dist(x_0,B)}
  \begin{aligned}
\dist(x_0,B) &\ge \dist(x_j,B) - \|x_0-x_j\|
\\
&\ge \dist(x_j,B) - 2j\dist(x_0,B)
\\
&\ge \dist(x_j,B) - 2(n-1)\dist(x_0,B).
\end{aligned}
\end{equation}
A combination of \eqref{n dist>dist} and \eqref{dist(x_0,B)}
yields
\begin{align*}
(2n-1)\dist(x_0,B) &\ge \frac{1-c}{2n}\left(\dist(x_0,A\cap B\cap \Lambda)-\|x_0-x_j\|\right)\\
&\ge \frac{1-c}{2n}\dist(x_0,A\cap B\cap \Lambda)- \frac{(1-c)(n-1)}{n}\dist(x_0,B).
\end{align*}
Hence
\begin{align*}
\dist(x_0,A\cap B\cap \Lambda)
\le \frac{2(2n^2-1-c(n-1))}{1-c}\dist(x_0,B)\quad \forall x_0\in A\cap \Ball_{\rho}(\xbar)\cap \Lambda.
\end{align*}
 This yields subtransversality of $\{A,B\}$ at $\bx$ relative to $\Lambda$ and $\srr \le \frac{2(2n^2-1-c(n-1))}{1-c}$ as claimed.
\end{proof}

The next statement is an immediate consequence of Proposition \ref{t:cp ncvx} and Theorem \ref{NonCon_Nec1+}.
\begin{corollary}[subtransversality is necessary and sufficient for linear monotonicity]\label{Nec_Suf_Sub}
Let $\Lambda\subset\Ebb$ be an affine subspace and let $A$ and $B$ be closed  subsets of $\Ebb$ that are
elementally subregular relative to $S\subset A\cap B\cap \Lambda$ at $\xbar\in S$
with constant $\varepsilon$ and neighborhood $\Ball_{\delta}(\xbar)\cap \Lambda$
 for all
$(a,v)\in \gph{\pncone{A}}$ with
 $a\in \Ball_{\delta}(\xbar)\cap \Lambda$.

Suppose that every sequence of
alternating projections with the starting point sufficiently close
 to $\xbar$ is contained in $\Lambda$.
All such sequences of
alternating projections
are linearly monotone with respect to $S$ with constant $c\in [0,1)$
if and only if the collection of sets is subtransversal at $\xbar$ relative to $\Lambda$ (with an adequate balance of quantitative constants).
\end{corollary}

The next technical lemma allows us formally avoid the restriction ``monotone'' in Theorem \ref{NonCon_Nec1+}.

\begin{lemma}\label{Tec_Lem3}
Let $(x_k)_{k\in\Nbb}$ be a sequence generated by $T_{AP}$
that converges R-linearly to $\xbar\in A\cap B$ with rate $c\in [0,1)$.
Then there exists a subsequence $(x_{k_n})_{n\in\Nbb}$ that is linearly monotone with respect to any
 set $S\subset A\cap B$ with $\xbar\in S$.
\end{lemma}
\begin{proof}
By definition of R-linear convergence, there is $\gamma<+\infty$ such that $\|x_k-\xbar\|\le \gamma c^k$ for all $k\in \Nbb$.
Let $S$ be any set such that $\xbar\in S\subset A\cap B$.
If $x_{k_0}:=x_0\notin S$, i.e., $\dist(x_{k_0},S)>0$, then there exists an iterate of $(x_k)_{k\in\Nbb}$
(we choose the first one) relabeled $x_{k_1}$ such that
\begin{align}\label{k_1}
\dist(x_{k_1},S) \le \|x_{k_1}-\xbar\| \le \gamma c^{k_1} \le c \dist(x_{k_0},S).
\end{align}
Repeating this argument for $x_{k_1}$ in place of $x_{k_0}$ and so on, we extract a subsequence $(x_{k_n})_{n\in \Nbb}$ satisfying
\[
\dist(x_{k_{n+1}},S) \le c \dist(x_{k_n},S)\quad \forall n\in\Nbb.
\]
The proof is complete.
\end{proof}

The above observation allows us to obtain the statement about necessary conditions for linear convergence of the alternating projections
algorithm which extends Theorem \ref{NonCon_Nec1+}.
Here, the index number
 $k_1$ depending on the sequence $(x_k)_{k\in\Nbb}$ will come into play in determining the constant of linear regularity.

\begin{thm}[subtransversality is necessary for linear convergence]\label{NonCon_Nec2}
Let $m\in \mathbb{N}$ be fixed and $c\in [0,1)$.
Let $\Lambda$, $A$ and $B$ be closed subsets of $\Ebb$
  and let $\xbar\in S\subset A\cap B\cap \Lambda$.
Suppose that any alternating projections sequence $(x_k)_{k\in\Nbb}$ starting in $A\cap \Lambda$ and sufficiently
close to $\xbar$ is contained in $\Lambda$, converges R-linearly to a point
  in $S$ with rate $c$, and the index $k_1\le m$ where $k_1$
  satisfies \eqref{k_1}.
Then the collection of sets $\{A,B\}$ is subtransversal at $\xbar$ relative to $\Lambda$ with constant
$\srr\le \frac{2m}{1-c}$.
\end{thm}
\begin{proof}
Let $\rho>0$ be so small that
 any alternating projections sequence starting in $A\cap \Ball_{\rho}(\xbar)\cap \Lambda$ converges 
 R-linearly to a point in $S$ with rate $c$.
Take any $x_0\in A\cap \Ball_{\rho}(\xbar)\cap \Lambda$ and generate an alternating projections sequence $(x_k)_{k\in\Nbb}$.
By Lemma \ref{Tec_Lem3}, there is a subsequence $(x_{k_n})_{n\in\Nbb}$ linearly monotone with respect to $S$ at rate $c$.
Then
\begin{align*}
\|x_{k_1}-x_{k_0}\| \ge \dist(x_{k_0},S)-\dist(x_{k_1},S) \ge (1-c)\dist(x_{k_0},S).
\end{align*}
Let $b_0\in P_Bx_0$ be such that $x_1\in P_Ab_0$ and note that $x_{k_0}=x_0$.
By the definition of the projection and $k_1\le m$ it follows that 
\begin{align*}
2m \dist(x_{0},B) = 2m \|b_0-x_0\| \ge \|x_{k_1}-x_{0}\| \ge (1-c)\dist(x_{0},S)\ge (1-c)\dist(x_0,A\cap B\cap \Lambda).
\end{align*}
Hence
\begin{align*}
\dist(x_0,A\cap B) \le \frac{2m}{1-c}\dist(x_0,B)\quad \forall
x_0\in A\cap \Ball_{\rho}(\xbar)\cap \Lambda.
\end{align*}
This yields subtransversality of $\{A,B\}$ at $\bx$ relative to $\Lambda$ and $\srr\le \frac{2m}{1-c}$ as claimed.
\end{proof}

The joining alternating projections sequence $(z_k)_{k\in\Nbb}$ given by \eqref{z_k} often plays a role as an
intermediate step in the analysis of alternating projections.
As we shall see, property of linear extendability itself can also be of interest when dealing
with the alternating projections algorithm, especially for nonconvex setting.
This observation can be seen for example in \cite{LewisMalick08,LewLukMal09,BauLukePhanWang13a,
NolRon16,DruIofLew15}.

\begin{thm}[subtransversality is necessary for linear extendability of subsequences]\label{NonCon_Nec1+_part2}
Let $\Lambda$,  $A$, and $B$ be closed subsets of $\Ebb$, let
 $\xbar\in A\cap B\cap\Lambda$, and let $1\le n\in \mathbb{N}$ and $c\in [0,1)$
be fixed.
Suppose that every alternating projections sequence $(x_k)_{k\in\Nbb}$ starting in $A\cap \Lambda$ and
sufficiently close to $\xbar$ has a subsequence of the form $(x_{j+nk})_{k\in\Nbb}$ for
some $j\in \{0,1,\ldots,n-1\}$ such that the joining sequence $(z_k)_{k\in\Nbb}$ given by \eqref{z_k}
 is a linear extension of $(x_{j+nk})$ on $\Lambda$
with frequency $2n$ and rate $c$.
Then the collection of sets $\{A,B\}$ is subtransversal at $\xbar$
relative to $\Lambda$ with constant $\srr\le \frac{2(2n-1-c(n-1))}{1-c}$.
\end{thm}

\begin{proof}
Let $\rho>0$ be so small that any alternating projections sequence starting in
$A\cap \Ball_{\rho}(\xbar)\cap \Lambda$ has a subsequence of the described form
which admits the joining
sequence as a linear extension on $\Lambda$ with frequency $2n$ and rate $c$.
Take any $x_0\in A\cap \Ball_{\rho}(\xbar)\cap \Lambda$.
Let us consider any alternating projections sequence $(x_k)_{k\in\Nbb}$ starting at $x_0$,
the corresponding joining sequence $(z_k)_{k\in\Nbb}$ and the subsequence $(x_{j+nk})_{k\in \Nbb}$.
Let $\xtilde\in \Lambda$ be the limit of $(z_k)_{k\in\Nbb}$ as verified in
 Proposition \ref{t:Com_LinCon_point}.

On one hand,
\begin{equation}\label{e1:le}
  \begin{aligned}
\dist(x_j,A\cap B\cap \Lambda) &\le \|x_j-\xtilde\| = \|z_{2j}-\xtilde\|
\\
&\le \sum_{i=2j}^{\infty}\|z_i-z_{i+1}\|
\le \frac{2n}{1-c}\|z_{2j}-z_{2j+1}\|
\\
& = \frac{2n}{1-c}\dist(z_{2j},B) = \frac{2n}{1-c}\dist(x_{j},B)\\
& \le \frac{2n}{1-c}\dist(x_0,B),
\end{aligned}
\end{equation}
where the last estimate follows from the nature of alternating projections.

On the other hand,
\begin{equation}\label{e2:le}
  \begin{aligned}
\dist(x_j,A\cap B\cap \Lambda)
&\ge \dist(x_0,A\cap B\cap \Lambda)-\|x_0-x_j\|\\
&\ge \dist(x_0,A\cap B\cap \Lambda)- 2(n-1)\dist(x_0,B),
\end{aligned}
\end{equation}
where the last estimate holds true since
\[
\norm{x_0-x_j} \le 2j\dist(x_0,B) \le 2(n-1)\dist(x_0,B).
\]

A combination of \eqref{e1:le} and \eqref{e2:le} then implies
\begin{align*}
\dist(x_0,A\cap B\cap \Lambda)
\le \frac{2(2n-1-c(n-1))}{1-c}\dist(x_0,B)\quad \forall x_0\in A\cap \Ball_{\rho}(\xbar)\cap \Lambda,
\end{align*}
which yields subtransversality of $\{A,B\}$ at $\bx$ relative to $\Lambda$ and $\srr \le \frac{2(2n-1-c(n-1))}{1-c}$ as claimed.
\end{proof}

In general, subtransversality is not a sufficient condition for
an alternating projections sequence
 to converge to a point in the intersection of the sets.
For example, let us define the function $f:[0,1]\to \R$ by $f(0)=0$ and on each interval of form $(1/2^{n+1},1/2^{n}]$,
\begin{equation*}
f(t)=
\left\{
  \begin{array}{ll}
    -t+1/2^{n+1}, & \mbox{if }\; t\in (1/2^{n+1},3/2^{n+2}],\\
   \;\; t-1/2^n, & \mbox{if }\; t\in (3/2^{n+2},1/2^{n}],
  \end{array}
\right.
\quad (\forall n\in \Nbb)
\end{equation*}
and consider the sets: $A=\gph f$ and $B=\{(t,t/3)\mid t\in [0,1]\}$ and the point $\bx=(0,0)\in A\cap B$ in $\R^2$.
Then it can be verified that the collection of sets $\{A,B\}$ is subtransversal at $\bx$ while the
alternating projections method gets stuck at points $(1/2^n,0) \notin A\cap B$.
\bigskip

To conclude this section, we show that the property of subtransversality of the
collection of sets has been imposed either explicitly or implicitly in all existing linear
convergence criteria for the method of alternating projections that we are aware of.
The next proposition catalogs existing linear convergence criteria for alternating projections
which complement Proposition \ref{t:cp ncvx}.

\begin{proposition}[R-linear convergence of nonconvex alternating projections]\label{t:AP}
Let $A$ and $B$ be closed and $\bx\in A\cap B$. The collection of sets is denoted $\{A,B\}$.
All alternating projections iterations starting sufficiently close to $\bx$
 converge R-linearly to some point in $A\cap B$ if one of the following conditions holds.
 \begin{enumerate}[\rm (i)]
  \item\label{t:LM} \cite[Theorem 4.3]{LewisMalick08} $A$ and $B$ are smooth manifolds 
  around $\xbar$ and $\{A,B\}$ is transversal at $\xbar$.
  \item\label{t:DIL} \cite[Theorem 6.1]{DruIofLew15} $\{A,B\}$ is intrinsically transversal at $\xbar$.
  \item\label{t:LLM} \cite[Theorem 5.16]{LewLukMal09} $A$ is super-regular at $\xbar$ and $\{A,B\}$ is transversal at $\xbar$.
  \item\label{t:BLPW} \cite[Theorem 3.17]{BauLukePhanWang13b} $A$ is $(B,\varepsilon,\delta)$-regular at
$\xbar$ and the $(A,B)$-qualification condition holds at $\xbar$.
  \item\label{t:NollRon} \cite[Theorem 2]{NolRon16}
$A$ is $0$-H\"older regular relative to $B$ at $\xbar$ and $\{A,B\}$ intersects separably at $\xbar$.
 \end{enumerate}
\end{proposition}

It can be recognized without much effort that under any item
of Proposition \ref{t:AP}, the sequences generated by alternating projections starting sufficiently close to $\bx$ are actually linearly extendible.

\begin{proposition}[ubiquity of subtransversality in linear convergence criteria]\label{Sub_appear}\
Suppose than one of the conditions \ref{t:LM}--\ref{t:NollRon} of Proposition \ref{t:AP} is satisfied.
Then for any alternating projections sequence $(x_k)_{k\in\Nbb}$
starting
sufficiently close to $\xbar$, the corresponding joining sequence $(z_k)_{k\in\Nbb}$ given by \eqref{z_k}
 is a linear extension of $(x_k)_{k\in\Nbb}$ with frequency $2$ and rate $c\in [0,1)$.
\end{proposition}
\begin{proof}
The statement can
 be observed directly from the key estimates that were used in proving the corresponding
convergence criterion. In fact, all the criteria listed in Proposition \ref{t:AP} essentially were obtained
from the same fundamental estimate which we named \emph{linear extendability} in this paper.
\end{proof}

Taking Theorem
\ref{NonCon_Nec1+_part2} into account we conclude that subtransversality of the collection of sets $\{A,B\}$ at $\xbar$
is a consequence of each item listed in Proposition \ref{t:AP}.
This observation gives some insights about relationships between various regularity notions of collections of sets and has been
formulated partly in \cite[Theorem 6.2]{DruIofLew15} and \cite[Theorem 4]{KruLukNgu16}.
Hence, the subtransversality property lies at the foundation of all
linear convergence criteria for the method of alternating projections for both convex and nonconvex sets appearing
in the literature to this point.

\section{Application: alternating projections with convexity}\label{s:AP_with_Convex}
In the convex setting, statements with sharper convergence rate estimates are possible.  This is
the main goal of the present section.  Note that a convex set is
elementally regular at all points in the set for all normal vectors with constant $\varepsilon=0$
and neighborhood $\Ebb$ \cite[Proposition 4(vii)]{KruLukNgu17}.  We can thus, without loss of
generality, remove the restriction to the subset
$\Lambda$ that is omnipresent in the nonconvex setting.
We also write $P_A x$ and $P_Bx$ for the projections since the
projectors are single-valued.

The next technical lemma is fundamental for the subsequent analysis.

\begin{lemma}[non-decreasing rate]\label{BasLem} Let $A$ and $B$ be two closed convex sets in $\Ebb$.
We have
\begin{equation}\label{BasEst}
\norm{P_BP_AP_Bx - P_AP_Bx}\cdot \norm{P_Bx - x} \ge \norm{P_AP_Bx - P_Bx}^2\quad \forall x\in A.
\end{equation}
\end{lemma}
\begin{proof}
Using the basic facts of the projection operators on a closed and convex sets, we obtain
\begin{align*}
\norm{P_AP_Bx - P_Bx}^2 &\le \ip{x - P_Bx}{P_AP_Bx - P_Bx}
\\
&= \ip{x - P_Bx}{P_AP_Bx - P_BP_AP_Bx} + \ip{x - P_Bx}{P_BP_AP_Bx - P_Bx}
\\
&\le \norm{x - P_Bx}\cdot \norm{P_AP_Bx - P_BP_AP_Bx}.
\end{align*}
The last
estimate holds true since the second term on the previous line is non-positive.
\end{proof}

Lemma \ref{BasLem} implies that for any sequence $(x_k)_{k\in\Nbb}$ of alternating projections  for convex sets, the rate
$\frac{\|x_{k+1}-x_k\|}{\|x_{k}-x_{k-1}\|}$ is nondecreasing when $k$ increases.
This allows us to deduce the following fact about the algorithm.

\begin{thm}[lower bound of complexity]\label{By_Pro_The1} Consider the alternating projections algorithm
for two closed convex sets $A$ and $B$ with a nonempty intersection.
Then one of the following statements holds true.
\begin{enumerate}
\item\label{5.2(i)} The alternating projections method finds a
solution
 after one iteration.
\item\label{5.2(ii)} Alternating projections will not reach a solution after any finite number of iterations.
\end{enumerate}
\end{thm}

\begin{proof}
If the starting point is actually in $A\cap B$, the proof becomes trivial.
Let us consider the case that $x_0\in A\setminus B$.
Suppose that the alternating projections method does not find a solution
after one iterate, that is, $x_1=P_AP_Bx_0 \notin A\cap B$. In other words,
 we suppose that scenario \eqref{5.2(i)}
does not occur and prove the validity of scenario \eqref{5.2(ii)}.

In this case, it holds that $P_Bx_0 \in B\setminus A$ as $x_1=P_AP_Bx_0 \notin A\cap B$.
As a result, $\|x_1-P_Bx_0\|>0$.
We also claim that $\|x_1-P_Bx_0\|<\|P_Bx_0-x_0\|$.
Indeed, suppose otherwise that $\|x_1-P_Bx_0\| = \|P_Bx_0-x_0\|$ (note that $\|x_1-P_Bx_0\| \le \|P_Bx_0-x_0\|$ by definition of projection). 
Then $\|P_Bx_0-x_0\|= \dist(P_Bx_0,A)$, which implies that $x_0$ is a fixed point of $P_AP_B$, $x_0 = P_AP_Bx_0$.
This contradicts the fact that $x_0\in A\setminus B$ and $A\cap B \neq \emptyset$ since any alternating
projections sequence for convex sets with nonempty intersection will converge to a point in the intersection \cite{BauBorSIREV96}.
Hence, we have checked that
\[
0<\|x_1-P_Bx_0\|<\|P_Bx_0-x_0\|.
\]
Then the following constant is well defined:
\begin{equation}\label{sqrt:c}
  \sqrt{c}:=\frac{\|x_1-P_Bx_0\|}{\|P_Bx_0-x_0\|} \in (0,1).
\end{equation}
Using Lemma \ref{BasLem} we get
\begin{align*}
\frac{\dist(x_1,B)}{\|x_1-P_Bx_0\|} = \frac{\norm{P_Bx_1-x_1}}{\|x_1-P_Bx_0\|} \ge \frac{\|x_1-P_Bx_0\|}{\|P_Bx_0-x_0\|} = \sqrt{c}.
\end{align*}
Hence
\begin{align*}
\dist(x_1,B)\ge \sqrt{c}\|x_1-P_Bx_0\| = c\|P_Bx_0-x_0\| = c\dist(x_0,B) >0.
\end{align*}
Applying Lemma \ref{BasLem} consecutively, we obtain
\begin{align*}
\dist(x_k,B) \ge c^{k}\dist(x_0,B)>0\quad \forall k\in\Nbb.
\end{align*}
This particularly implies that $x_k\notin A\cap B$ for any natural number $k\in \mathbb{N}$, and
the proof is complete.
\end{proof}

\begin{remark}
In contrast to Theorem \ref{By_Pro_The1} for convex sets, there are simple examples of nonconvex sets such that for any
given number $n\in \mathbb{N}$, the alternating projections method will find a solution after exactly $n$ iterates.
For instance, let us consider a geometric sequence $z_k = \left(\frac{1}{3}\right)^kz_0$ where $0\neq z_0\in \Ebb$.
For any number $n\in \mathbb{N}$, one can construct the two finite sets
 by $A:=\{z_{2k}\mid k=0,1,\ldots,n\}$ and $B:=\{z_{2n}\} \cup \{z_{2k+1} \mid k=0,1,\ldots,n-1\}$.
Then the alternating projections method starting at $z_0$ will find the unique solution $z_{2n}$ after exactly $n$ iterates.
\end{remark}

\begin{thm}[necessary and sufficient condition: local version]\label{BasThe_Loc}
Let $A$ and $B$ be closed convex sets and $\xbar\in A\cap B$.
If the collection of sets $\{A, B\}$ is subtransversal at $\xbar$ with constant $\srr < +\infty$, then for any number
$c\in (1-\srr^{-2},1)$, all alternating projections sequences starting sufficiently close to $\xbar$ are linearly
monotone with respect to $A\cap B$ with rate not greater than $c$.

Conversely, if there exists a number $c\in [0,1)$ such that every alternating projections iteration starting sufficiently
close to $\xbar$ converges R-linearly to some point in
$A\cap B$
with rate not greater than $c$, then the collection of sets $\{A, B\}$ is subtransversal at $\xbar$ with
constant $\srr \le \frac{1}{1-c}$.
\end{thm}
\begin{proof}
The first implication is an adaption of \cite[Corollary 3.13(c)]{HesLuk13} to the terminology of this paper.

We now prove the converse implication.
Let $\rho>0$ be so small that every alternating projections iteration starting in $B_{\rho}(\xbar)$ 
converges R-linearly to a point in $A\cap B$
 with rate not greater than $c$.
Take any $x_0\in A\cap B_{\rho}(\xbar)$.
Let us consider the alternating projections sequence $(x_k)_{k\in\Nbb}$ starting at $x_0$ and converging
R-linearly to $\xtilde\in A\cap B$ with rate not greater than $c$.
By definition of R-linear convergence, there is a number $\gamma>0$ such that
\begin{align}\label{l.c.}
\|x_k-\xtilde\|\le \gamma c^k\quad \forall k\in \Nbb.
\end{align}
Taking Theorem \ref{By_Pro_The1} into account, we consider the two possible cases as follows.

\emph{Case 1.} The alternating projections method finds a solution after one iterate, $x_1=P_AP_Bx_0\in A\cap B$.
Lemma \ref{BasLem} yields
\[
\norm{x_1 - P_Bx_0}^2 \le \norm{P_Bx_1 - x_1}\cdot \norm{P_Bx_0 - x_0} = 0.
\]
This implies that $P_Bx_0 = x_1\in A\cap B$ and as a result,
\begin{align}\label{Cas1}
\dist(x_0,A\cap B)\le \|x_0-P_Bx_0\| = \dist(x_0,B).
\end{align}

\emph{Case 2.} The alternating projections do not reach a solution after any finite number of iterates.
We will make use of the joining sequence $(z_k)_{k\in\Nbb}$ given by \eqref{z_k}.
Since $P_A$ and $P_B$ are firmly nonexpansive, the sequence $(z_k)_{k\in\Nbb}$ is Fej\'er monotone with respect to $A\cap B$.
Then it follows that
\begin{align}\label{Gap1}
\|z_{k+1}-z_k\| = \max\{\dist(z_k,A),\dist(z_k,B)\} \le 
\norm{z_k-\xtilde}\le \frac{\gamma}{\sqrt{c}} \sqrt{c}^k\quad \forall k\in\Nbb.
\end{align}
We claim that
\begin{equation}\label{Cla1}
\|z_{k+1}-z_{k}\| \le \sqrt{c}\|z_k-z_{k-1}\|\quad \forall k\in \mathbb{N}.
\end{equation}
Suppose to the contrary that there exists a natural number $p\ge 1$ such that $\|z_{p+1}-z_{p}\|> \sqrt{c}\|z_p-z_{p-1}\|$.
Choose a number $\theta>\sqrt{c}$ such that $\|z_{p+1}-z_{p}\|\ge \theta\|z_p-z_{p-1}\|$.
Then applying Lemma \ref{BasLem}, we get
\begin{align*}
\|z_{k+1}-z_{k}\| \ge \theta^{k-p+1}\|z_p-z_{p-1}\|\quad \forall k\ge p,\; k\in\Nbb.
\end{align*}
This together with \eqref{Gap1} implies that for all natural number $k\ge p$,
\begin{align*}
& \frac{\gamma}{\sqrt{c}} \sqrt{c}^k \ge \theta^{k-p+1}\|z_{p}-z_{p-1}\|\\
\Leftrightarrow & \frac{\gamma}{\sqrt{c}} \sqrt{c}^k \ge \sqrt{c}^{k-p+1}\left(\frac{\theta}{\sqrt{c}}\right)^{k-p+1}\|z_{p}-z_{p-1}\|\\
\Leftrightarrow & \frac{\gamma}{\sqrt{c}} \sqrt{c}^{p-1} \ge \left(\frac{\theta}{\sqrt{c}}\right)^{k-p+1}\|z_{p}-z_{p-1}\|.
\end{align*}
Letting $k\to +\infty$, the last inequality leads to a contradiction since $\frac{\theta}{\sqrt{c}}>1$.
Hence, \eqref{Cla1} has been proved.

Now, using \eqref{Cla1} and the firm nonexpansiveness of $P_A$ and $P_B$, we obtain that
\begin{equation}\label{Cas2}
  \begin{aligned}
\dist(x_0,A\cap B) \le & \|x_0-\xtilde\| \le \sum_{j=0}^{\infty}\|z_{2j+2}-z_{2j}\|  \le \sum_{j=0}^{\infty}\|z_{2j+1}-z_{2j}\|\\
\le & \sum_{j=0}^{\infty}\|z_1-z_0\|\sqrt{c}^{2j} \le \frac{1}{1-c}\|z_1-z_0\| = \frac{1}{1-c}\dist(x_0,B).
\end{aligned}
\end{equation}

A combination of \eqref{Cas1} and \eqref{Cas2}, which respectively correspond to the two cases, implies that
\begin{align*}
\dist(x_0,A\cap B) \le \frac{1}{1-c}\dist(x_0,B)\quad \forall x_0\in A\cap B_{\rho}(\xbar).
\end{align*}
Hence $\{A,B\}$ is subtransversal at $\xbar$ and the constant $\srr \le \frac{1}{1-c}$ as claimed.
\end{proof}

The next theorem provides a global version of Theorem \ref{BasThe_Loc}.

\begin{thm}[necessary and sufficient condition: global version]\label{BasThe_Glo}
Let $A$ and $B$ be closed convex sets with nonempty intersection.
If the collection of sets $\{A, B\}$ is subtransversal at every point of (the boundary of) $A\cap B$
with constants bounded from above by $\kappa <+\infty$,
then for any number $c \in (1-\kappa^{-2},1)$,
every alternating projections iteration
 converges R-linearly to a point in
$A\cap B$ with rate not greater than $c$.

Conversely, if there exists a number $c\in [0,1)$ such that
every alternating projections sequence eventually converges R-linearly to a point in $A\cap B$
 with rate not greater than $c$, then the collection of sets $\{A,B\}$ is globally subtransversal
with constant $\kappa \le \frac{1}{1-c}$,
that is,
\begin{equation}\label{global_subtran}
  \dist(x,A\cap B) \le \frac{1}{1-c}\dist(x,B)\quad \forall x\in A.
\end{equation}
\end{thm}

\begin{proof}
We prove the first implication.
Let us take any point $x_0\in \Ebb$ and
consider the alternating projections sequence $(x_k)_{k\in\Nbb}$
starting at $x_0$.
It suffices to consider only the case that the alternating projections method does not find a solution after one iterate.
It is well known that $(x_k)_{k\in\Nbb}$ converges to some point $\xtilde\in \bd(A\cap B)$ \cite{BauBorSIREV96}.
Hence, after a finite number, say $p$, of iterates, the iterate $x_p$ must be sufficiently close to $\xtilde$.
Using the assumption that $\{A,B\}$ is subtransversal at $\xtilde$ with constant $\srrt \le \kappa$ and applying Theorem \ref{BasThe_Loc},
we deduce that the alternating projections sequence starting from $x_p$ converges R-linearly to $\xtilde$ with 
rate not greater than $c$.
On one hand, using \eqref{Cla1} for the alternating projections sequence
starting from $x_p$ and the corresponding joining sequence,
 we get
\begin{align}\label{Cla11}
\|z_{k+1}-z_{k}\|\le \sqrt{c}\|z_k-z_{k-1}\|\quad \forall k\ge p+1, k\in\Nbb.
\end{align}
On the other hand, applying Lemma \ref{BasLem}, we get
\begin{align}\label{Cla12}
\frac{\|z_{k+1}-z_{k}\|}{\|z_k-z_{k-1}\|}\le \frac{\|z_{p+2}-z_{p+1}\|}{\|z_{p+1}-z_{p}\|} \le \sqrt{c}\quad \forall k\le p, k\in\Nbb.
\end{align}
A combination of \eqref{Cla11} and \eqref{Cla12} yields the estimate \eqref{Cla1}.
Proposition \ref{t:Com_LinCon} then ensures that the joining sequence $(z_k)_{k\in\Nbb}$ converges R-linearly to $\xtilde$
 with rate not greater than $\sqrt{c}$.
This implies that the sequence $(x_k)_{k\in\Nbb}$ converges R-linearly to $\xtilde$ with rate not greater than $c$ as claimed.

We now prove the converse implication.
Suppose that every sequence
 of alternating projections eventually converges R-linearly to a point in $A\cap B$ with rate not greater than $c$.
We need to verify \eqref{global_subtran}.
Note that the estimate \eqref{global_subtran} is trivial for $x\in A\cap B$.
Let us take an arbitrary $x\in A\setminus B$ and consider the alternating projections sequence $(x_k)_{k\in\Nbb}$
starting at $x_0=x$. We consider the two possible cases as stated in Theorem \ref{By_Pro_The1}.

\emph{Case 1.} The alternating projections method finds a solution after one iterate. The argument for
\emph{Case 1} of the proof of Theorem \ref{BasThe_Loc} yields \eqref{global_subtran}.

\emph{Case 2.} The alternating projections method does not find a solution after any finite number of iterates.
Since $(x_k)_{k\in\Nbb}$ eventually converges R-linearly to a point $\xtilde\in A\cap B$ with rate not greater than $c$,
there exists a natural number $p\in \mathbb{N}$ and a constant $\gamma'>0$ such that
\begin{equation}\label{from_N}
\|x_k-\xtilde\| \le \gamma' c^{k-p} = \frac{\gamma'}{c^p}c^{k}\quad \forall k\ge p.
\end{equation}
Let us define the number
\begin{equation}\label{to_N}
\gamma :=\max\left\{
\frac{\gamma'}{c^p},
\frac{\|x_k-\xtilde\|}{c^k}:\; k=0,1,\ldots,p
\right\}>0.
\end{equation}
Combining \eqref{from_N} and \eqref{to_N} yields 
\begin{equation*}
\|x_k-\xtilde\| \le \gamma c^{k}\quad \forall k\in \Nbb.
\end{equation*}
The argument for \emph{Case 2} in the proof of Theorem \ref{BasThe_Loc} implies that the sequence
$(z_k)_{k\in\Nbb}$ defined at \eqref{z_k} satisfies
\[
\|z_{k+1}-z_k\| \le \sqrt{c}\|z_{k}-z_{k-1}\|\quad \forall k\in \mathbb{N}.
\]
From this condition, the estimate \eqref{global_subtran} is obtained by using the estimates at
\eqref{Cas2}.

The proof is complete.
\end{proof}

It is clear that Theorem \ref{BasThe_Loc} does not cover Theorem \ref{BasThe_Glo}. 
The following example also rules out the inverse inclusion.

\begin{eg}[Theorem \ref{BasThe_Glo} does not cover Theorem \ref{BasThe_Loc}]
Consider the convex function $f:\mathbb{R}\to \mathbb{R}$ given by
\begin{equation*}
f(t) = \begin{cases}
t^2, & \mbox{if } t \in [0,\infty),\\
0, & \mbox{if } t\in [-1,0),\\
-t-1, & \mbox{if } t \in (-\infty,-1).
\end{cases}
\end{equation*}
In $\mathbb{R}^2$, we define two closed convex sets $A:=\epi f$ and $B:=\mathbb{R}\times \mathbb{R}_{-}$ and a point $\xbar=(-1,0)\in A\cap B$.
Then the two equivalent properties (namely, transversality of $\{A,B\}$ at $\xbar$ and local linear convergence of $T_{AP}$ around $\bx$) 
involved in Theorem \ref{BasThe_Loc} hold true while the two global ones involved in Theorem \ref{BasThe_Glo} do not.
\end{eg}

To establish global convergence of a fixed point iteration, one normally needs some kind of global
regularity behavior of the fixed point set.
In Theorem \ref{BasThe_Glo}, we formally impose only subtransversality in order to deduce global R-linear convergence and vice versa.
Beside the global behavior of convexity, the hidden reason behind this seemingly contradictory phenomenon
is a well known fact about subtransversality of collections of convex sets.
We next deduce this result from the convergence analysis above. The proof is given for completeness.

\begin{corollary}\cite[Theorem 8]{Li97}\label{cor_reg}
Let $A$ and $B$ be closed and convex subsets of $\Ebb$ with nonempty intersection. The collection of sets $\{A,B\}$
is globally subtransversal, that is, there is a constant ${\kappa}< +\infty$ such that
\begin{equation}\label{global_subtran_2}
  \dist(x,A\cap B) \le \kappa\dist(x,B)\quad \forall x\in A,
\end{equation}
if and only if $\{A,B\}$ is subtransversal at every point in $\bd(A\cap B)$ with constants bounded from above by some $\overline{\kappa}<+\infty$.
\end{corollary}

\begin{proof}
$(\Rightarrow)$ This implication is trivial with $\overline{\kappa}=\kappa$.\\
$(\Leftarrow)$ Note that the estimate \eqref{global_subtran_2} is trivial for $x\in A\cap B$.
Let us take an arbitrary $x\in A\setminus B$ and consider the alternating projections sequence $(x_k)_{k\in\Nbb}$ starting at $x_0=x$.
Take any number $c\in \paren{1-\overline{\kappa}^{-2},1}$
The argument in the first part of Theorem \ref{BasThe_Glo} implies that $(x_k)_{k\in\Nbb}$ converges
R-linearly to some $\xtilde\in A\cap B$ with rate $c$ and
\begin{align*}
\dist(x,A\cap B)\le \frac{1}{1-c}\dist(x,B).
\end{align*}
By letting $c\downarrow 1-\overline{\kappa}^{-2}$ in the above inequality, we obtain \eqref{global_subtran_2}
with $\kappa = \overline{\kappa}^2$.

The proof is complete.
\end{proof}

The convergence counterpart of Corollary \ref{cor_reg} can also be of interest.

\begin{corollary}\label{cor_con}
Let $(x_k)_{k\in\Nbb}$ be an alternating projections sequence for two closed convex subsets of $\Ebb$
with nonempty intersection and $c\in [0,1)$. If there exists a natural number
$p\in \mathbb{N}$ such that $\|x_k-\xtilde\|\le \gamma c^k$ for all $k\ge p$, then $\|x_k-\xtilde\|\le \gamma c^k$ for all $k\in \mathbb{N}$.
\end{corollary}

We emphasize that the two statements in Corollary \ref{cor_con} are always equivalent
(by the argument for the second part of
Theorem \ref{BasThe_Glo}) if the constant $\gamma$ is not required to be the same.
However, this requirement becomes important when one wants to estimate global rate of
convergence via the local rate of convergence.
The next statement can easily be observed as a by-product via the proof of Theorem \ref{BasThe_Loc}.

\begin{proposition}[equivalence of linear monotonicity and R-linear convergence]\label{u_R}
For sequences of alternating projections between convex sets, R-linear convergence and linear monotonicity of the
sequence of iterates are equivalent.
\end{proposition}

The next statement can serve as a motivation for Definition \ref{Uni_Lin_Con_to_point}.

\begin{proposition}[Q-linear convergence implies linear extendability]\label{Q_R}
Let $(x_k)_{k\in\Nbb}$ be a sequence of alternating projections
  for two closed convex sets $A,B \subset \Ebb$ with nonempty intersection.
If $(x_k)_{k\in\Nbb}$ converges Q-linearly to a point $\xtilde\in A\cap B$ with rate $c\in [0,1)$, then
$(x_k)_{k\in\Nbb}$ is linearly extendible with frequency $2$ and rate $c$, and the corresponding joining sequence
$(z_k)_{k\in\Nbb}$
 is such a linear extension sequence.
\end{proposition}

Before proving this,  we first establish  the following technical fact.
\begin{lemma}\label{Tec_Lem_2} Let $A$ and $B$ be two closed convex sets in $\Ebb$ with
nonempty intersection.
We have
\begin{equation}\label{BasEst_2}
\|P_Ba-x\|\|P_Ba-a\|\ge \|a-x\|\|P_AP_Ba-P_Ba\|\quad \forall a\in A,\;\forall x\in A\cap B.
\end{equation}
\end{lemma}
\begin{proof}[of Lemma \ref{Tec_Lem_2}.]
Denote $b=P_Ba$ and $a_+=P_AP_Ba$.
It suffices to consider the two cases as follows.

\emph{Case 1.} $\|b-x\|=0$ or $\|b-a\|=0$. This implies that $b\in A\cap B$, which in turn implies that $a_+=b$.
Hence, inequality \eqref{BasEst_2} is satisfied.

\emph{Case 2.} Both sides of \eqref{BasEst_2} are strictly positive.
Let $a'$ be the projection of $b$ on the line (segment, equivalently since $\ip{x-b}{a-b} \le 0$)
 joining $x$ and $a$.
The elementary geometry for triangles $\Delta xba$ and $\Delta xba'$, respectively, yields
\begin{align*}
\|b-a\| \geq \|a-x\|\sin\angle(b-x,a-x)>0,\\
\|a'-b\| = \|b-x\|\sin\angle(b-x,a-x)>0.
\end{align*}
This implies
\begin{align*}
\|b-x\|\|b-a\| \geq \|b-x\|\|a-x\|\sin\angle(b-x,a-x)=\|a-x\|\|a'-b\|.
\end{align*}
Inequality \eqref{BasEst_2} now follows since, by convexity of $A$, $a'\in A$, and by
 definition of the projector,
\begin{align*}
\|a'-b\| \ge \dist(b,A) = \|a_+-b\|.
\end{align*}
This proves Lemma \ref{Tec_Lem_2}.
\end{proof}

We conclude with the proof of Proposition \ref{Q_R}.

\begin{proof}[of Proposition \ref{Q_R}.]
It suffices to prove that the sequence $(z_k)_{k\in\Nbb}$ given in \eqref{z_k} satisfies
\[
\|z_{k+2}-z_{k+1}\|\le \sqrt{c}\|z_{k+1}-z_k\|\quad \forall k\in \Nbb.
\]
We will prove this by way of contradiction. Suppose otherwise that there exists some $p\in \Nbb$ such that
$$
\|z_{p+2}-z_{p+1}\|>\sqrt{c}\|z_{p+1}-z_{p}\|.
$$
We can assume $p=2k$ without loss of generality.
By Lemma \ref{BasLem} we get
\[
\frac{\|z_{2k+3}-z_{2k+2}\|}{\|z_{2k+2}-z_{2k+1}\|}
\ge \frac{\|z_{2k+2}-z_{2k+1}\|}{\|z_{2k+1}-z_{2k}\|}>\sqrt{c}.
\]
Lemma \ref{Tec_Lem_2} then implies
\begin{align*}
\frac{\|x_{k+1}-\xtilde\|}{\|x_{k}-\xtilde\|}
&=\frac{\|z_{2k+2}-\xtilde\|}{\|z_{2k}-\xtilde\|}
=\frac{\|z_{2k+2}-\xtilde\|}{\|z_{2k+1}-\xtilde\|}
\frac{\|z_{2k+1}-\xtilde\|}{\|z_{2k}-\xtilde\|}\\
&\ge \frac{\|z_{2k+3}-z_{2k+2}\|}{\|z_{2k+2}-z_{2k+1}\|}\frac{\|z_{2k+2}-z_{2k+1}\|}{\|z_{2k+1}-z_{2k}\|} >c.
\end{align*}
This contradicts Q-linear convergence of $(x_k)_{k\in\Nbb}$ to $\xtilde$ with rate $c$, and the proof is complete.
\end{proof}


\end{document}